\documentclass[11pt, english]{article}
\usepackage{babel}
\usepackage[latin1]{inputenc}
\usepackage{amsthm}
\usepackage{amsfonts,amsmath,amstext,amsbsy,euscript,amssymb}
\theoremstyle{plain}
\newtheorem{theorem}{Theorem}%[section]
\newtheorem{lemma}{Lemma}
\usepackage{graphicx}
\usepackage{hyperref}
\usepackage[font=small,skip=-4pt]{caption}
\usepackage{pgf}
\usepackage{tikz}
\usetikzlibrary{arrows}
\usetikzlibrary{decorations.pathmorphing}
\usetikzlibrary{backgrounds}
\usetikzlibrary{positioning}
\usetikzlibrary{fit}
\renewcommand{\le}{\leqslant}
\renewcommand{\ge}{\geqslant}
\newcommand{\CA}{\text{CA}}
\renewcommand{\top}{\text{top}}
\newcommand{\base}{\text{base}}
\newcommand{\support}{\text{support}}
\newcommand{\subsupport}{\text{subsupport}}

\begin{document}

\colorlet{darkgreen}{black!20!green}
\colorlet{darkblue}{blue!85!white}
\colorlet{lg}{white!95!yellow}
\colorlet{lightgreen}{blue!10!green}
\colorlet{lightgreen2}{white!85!green}
\colorlet{darkred}{orange!30!red}
\colorlet{darkbrown}{brown!80!black}

\title{Hopeful windows and fractals in cellular automata and combinatorial games}
\author{Urban Larsson\footnote{Department of Mathematics and Statistics, Dalhousie University, Canada; supported by the Killam Trust; urban031@gmail.com}}
\date{\today}

\maketitle
\begin{abstract}
This paper studies 2-player impartial combinatorial games, where the outcomes correspond to updates of cellular automata (CA) which generalize Wolfram's elementary rule 60 and rule 110 (Cook 2004). The games extend the class of \emph{triangle placing games} (Larsson 2013) where at each stage of the game the previous player has the option to block certain hopeful moves of the next player. We also study fractals and partial convergence in a subclass of the CA.
\end{abstract}
\section{Introduction}
In this paper we study a generalization of Wolfram's elementary cellular automata (CA) rules 60 and 110 \cite{C04} \cite{W02} to a class of CA, whose evolution diagrams are equivalent to the outcomes of 2-player impartial combinatorial games \cite{BCG82} with a blocking maneuver \cite{SS02}. An \emph{impartial game} is a 2-player game where both players have perfect information, the same options (same set of valid moves) at all times, and there is no element of chance involved in game play. The current player is called the \emph{next} player and the other player is called the \emph{previous} player. The \emph{outcome class} of an impartial game is the classification of a game position as a Next (N) player win or a Previous (P) player win. Typically, as these games are finite, the game tree is written in its entirety and all terminal positions are defined as P-positions (since the next player cannot move). To determine the winner of the current position, we recursively backtrack up the tree. For more information, see \cite{BCG82}. 

The game we consider throughout this paper is a triangle placement game, first examined in~\cite{L13}, with the additional option of a blocking maneuver. We describe the game informally here. Two players alternate to place right angle isosceles triangles, where the right angle is at the base and right justified. The top of the current triangle must be played within the \emph{support} of the base of the previous triangle. The support is an invisible strip directly underneath the placed triangle, determined by two nonnegative parameters $\ell$ (left) and $r$ (right). The game ends at a predetermined horizontal level, where some obstacle(s) have been placed, occupying at least a single cell. 

The game is hard to solve in general, and it belongs to the family of undecidable games, because of the equivalence of one of its member with the rule 110 CA, which is undecidable \cite{C04}. The triangles are discrete, and may be of any positive size, but never cover an obstacle at the terminal horizontal level, or go below this level. For more information on the triangle placement game see~\cite{L13}.
The paper is organized as follows. Section~\ref{sec: CA} defines a cellular automaton with new parameters. Section~\ref{sec:game} examines the extended triangle placing game with blocking moves. Section~\ref{sec: correspondence} explores the correspondence between CA and the triangle placing game from Section~\ref{sec:game}. Lastly, in Section~\ref{sec:frac} we prove a fractal behavior and partial convergence of limit diagrams.

\section{The cellular automaton}\label{sec: CA}
We define the state of a doubly infinite one-dimensional cellular automaton $\CA(\cdot ,t)\in {\{0,1\}^\mathbb Z}$, at  time $t\in \mathbb Z_{\ge 0}$. Consider an initial configuration: for all $x\in \mathbb Z$, $\CA(x,0)\in \{0,1\}$. Let $\Gamma , L, R, B\in \mathbb Z_{\ge 0}$, $\Gamma\ge 2$ and, given $\Gamma$, define a function, $\Delta=\Delta(L,R)=\Gamma+L+R>B$. For all $t\in \mathbb Z_{> 0}$, $\CA(\cdot, t)$ is defined via $\CA(\cdot, t-1)$, by the following update function. 
\[
  \CA(x,t)=\begin{cases}
               0, \text{ if }  \hspace{.2 cm}  \CA(x-\Gamma+1, t-1)+\ldots +\CA(x,t-1)=0, \text{ or}\\ 
            \hspace{1.1 cm} \CA(x-\Gamma+1-L, t-1)+\ldots + \CA(x+R, t-1)\ge \Delta - B\\
              1, \text{ otherwise.}
            \end{cases}
\]
 We use the notation $\CA=\CA_{\Gamma, L, R, B}(x,t)$ when $(x,t)$ ranges over all cells of the initial condition $t=0$ (at the lowest level) and updates thereafter, thus obtaining 2-d diagrams, exemplified in Figure~\ref{fig:1}, and we refer to $B$ as the \emph{blocking number}.  
\begin{figure}[ht!]
\begin{center}
%\vspace{0.2 cm}
\includegraphics[width=0.22\textwidth]{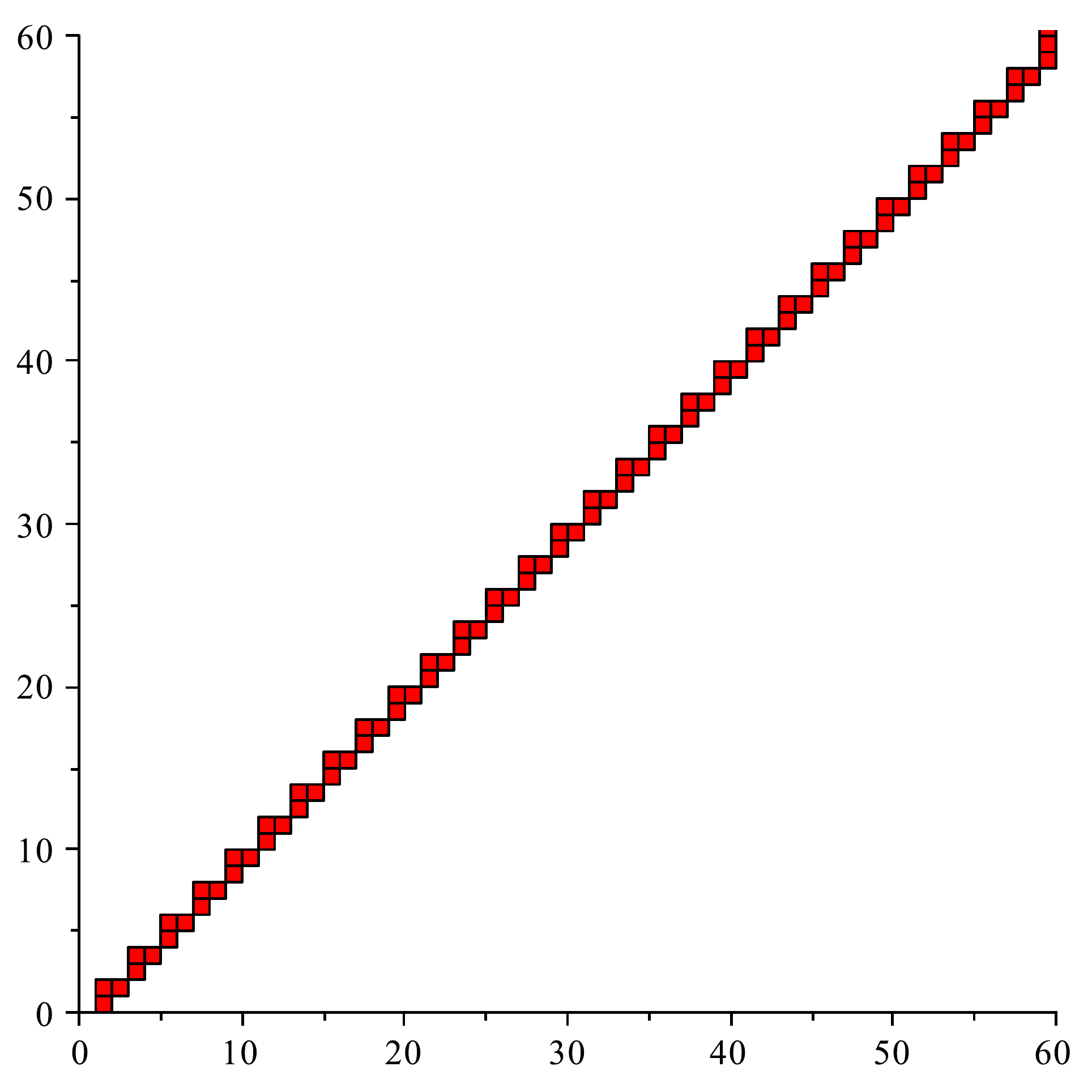}\hspace{.2 cm}
\includegraphics[width=0.22\textwidth]{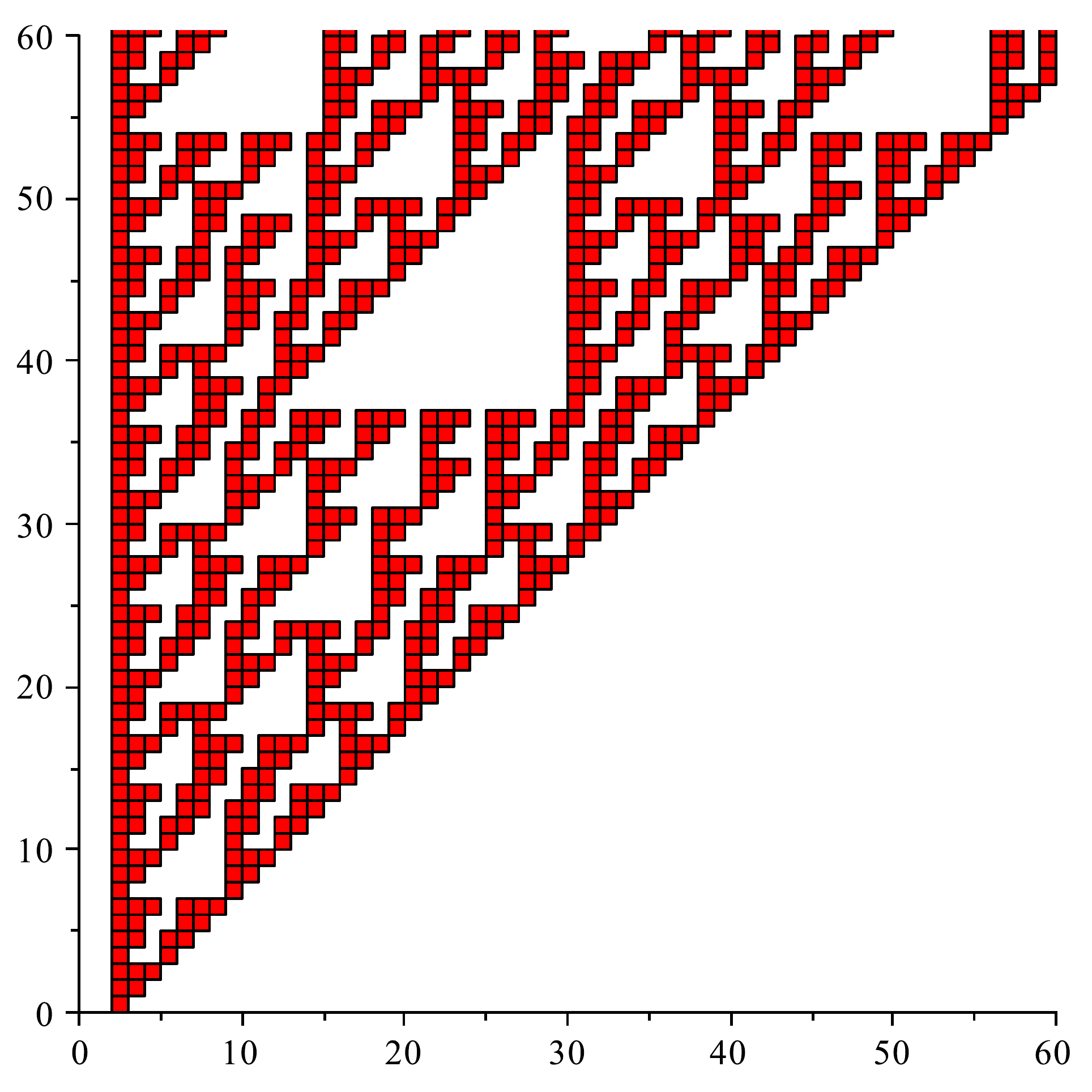}\hspace{.2 cm}
\includegraphics[width=0.255\textwidth]{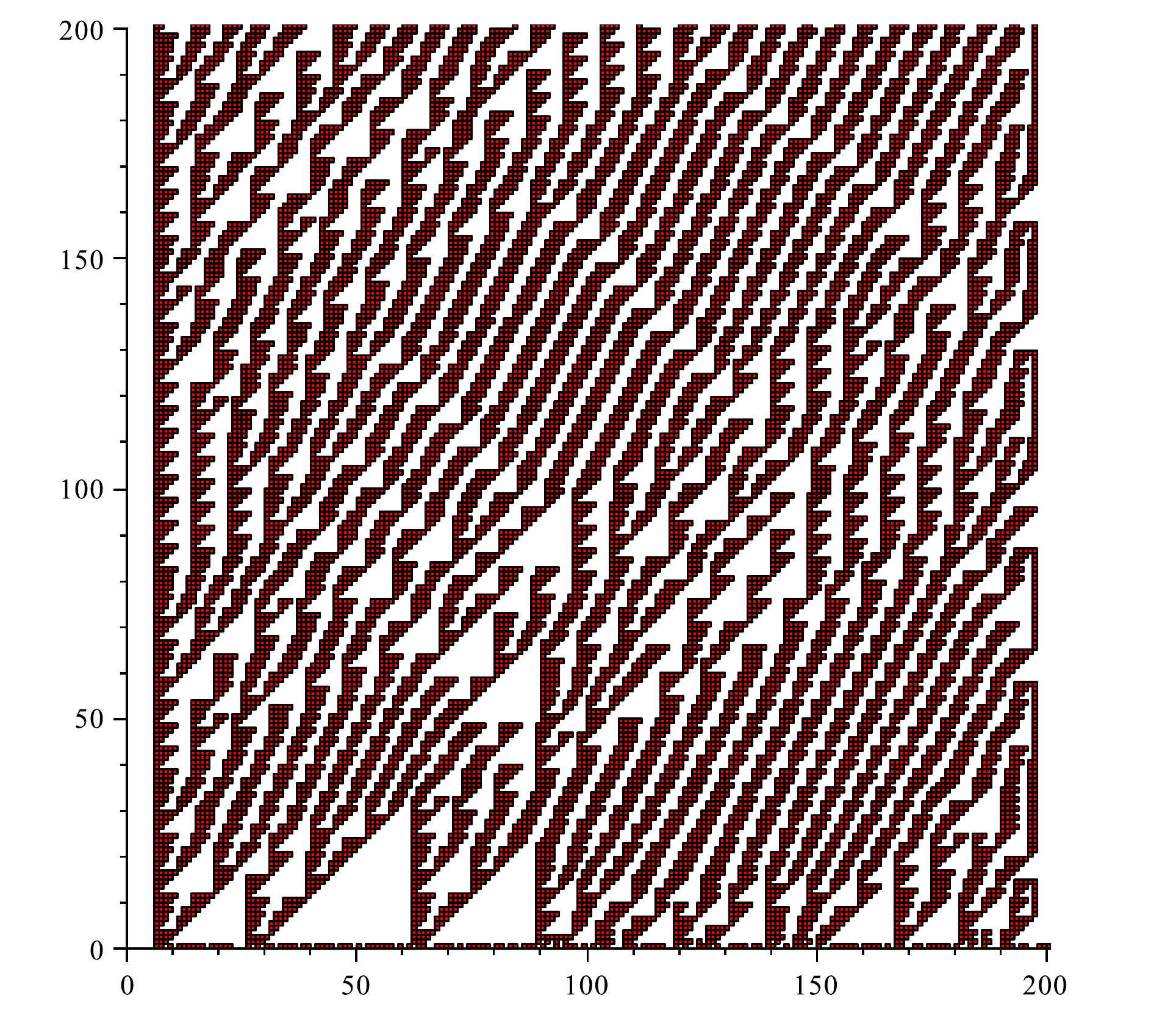}\hspace{.2 cm}
\includegraphics[width=0.22\textwidth]{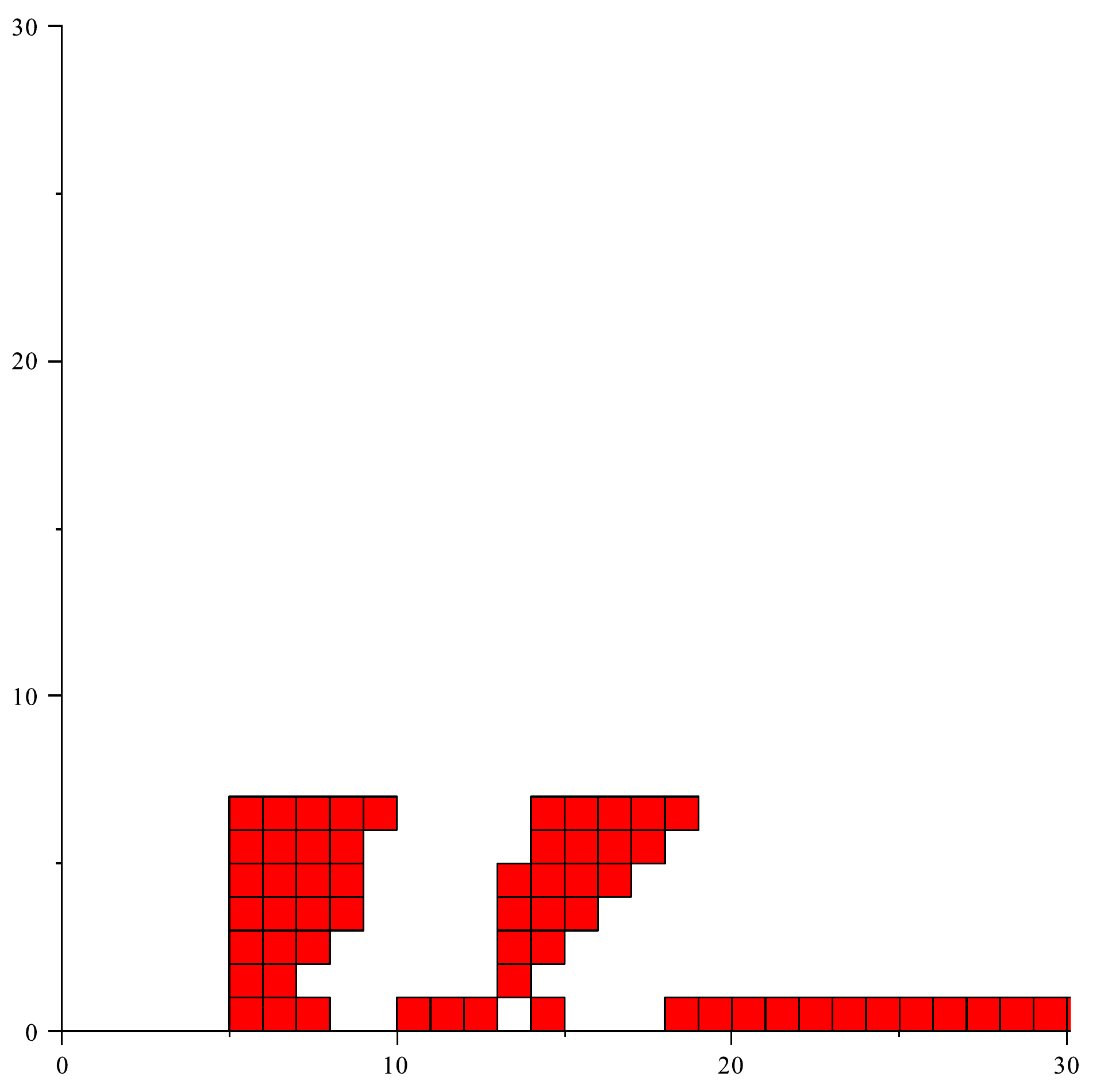}
\end{center}\caption{The $CA_{\Gamma, L, R, B}$ for $(\Gamma , L,R,B) = (2,0,1,1), (2,1,1,1), (2,3,3,3), (2,4,4,5)$ respectively. In the leftmost diagram (rule 110 with blocking number 1) and in the second one, the initial configuration is a single `1'; thereafter the initial condition is `random'; to the right, we wonder if this CA `dies out' for any initial configuration, a tendency for CA with relatively high blocking numbers.}\label{fig:1}
\end{figure}

We think of this as if the CA-bit in cell $x$ at time $t$ were defined by the values in two \emph{update windows}, $w_0=w_0(x,t)\subseteq w_1 = w_1(x, t)$, reading the `$(\Gamma, \Delta)$-neighborhoods'  at time $t-1$, as illustrated in Figure~\ref{fig:CAwindow} (with $\Gamma = 3, L=2, R=1$): only `0's in the inner (shaded--green) part of the window ($w_0$), or at most $B$ `0's in the full  window ($w_1)$, gives a `0', and otherwise the update will be a `1'. Note that the second condition of the definition of $\CA(x, t)$ is satisfied in the first two pictures in Figure~\ref{fig:CAwindow} (the distinction of the updates is important for the statement of Theorem~\ref{thm:1}). 

Suppose that $X$ is a finite bit-multiset (or bit-sequence). Then $|X|=|X|_0 + |X|_1$ counts its number of elements, $|X|_0$ counts its number of `0's and $|X|_1$ counts its number of `1's. Given $(x, t)$, the windows are the multisets $$w_0 = \{\CA(x-\Gamma+1, t-1), \ldots , \CA(x, t-1)\}$$ and $$w_1 = \{\CA(x-\Gamma+1-L, t-1), \ldots , \CA(x+R, t-1)\}.$$ 
Thus $|w_0| = \Gamma$, and $|w_1| = \Delta$. We get that $\CA(x,t) = 0$ if and only if $|w_0|_0 = |w_0|$ or $|w_1|_0\le B$. That is $\CA(x,t) = 1$ if and only if $|w_0|_1 > 0$ and $|w_1|_0 > B$. 

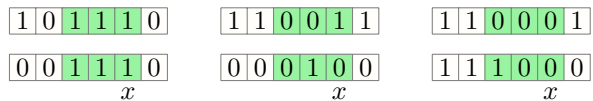
\begin{figure} [ht!]
\vspace{0.2 cm}
\begin{center}

\begin{tikzpicture} [scale = 0.35]
\foreach \x/\y in {0/0, 1/0, 2/0, 3/0, 4/0, 5/0 }
\filldraw[color = lg,opacity=.3] (\x, \y) rectangle (\x+1, \y+1);
\foreach \x/\y in { 2/0, 3/0, 4/0}
\filldraw[color = lightgreen,opacity=.4] (\x, \y) rectangle (\x+1, \y+1);
\draw[step=1cm,gray, thin] (0, 0) grid (6,1); 
\draw (0.5, .5) node {\small{$1$}};\draw (1.5, .5) node {\small{$0$}};\draw (2.5, .5) node {\small{$1$}};\draw (3.5, .5) node {\small{$1$}};\draw (4.5, .5) node {\small{$1$}};\draw (5.5, .5) node {\small{$0$}};
\end{tikzpicture}\hspace{.5 cm}
\begin{tikzpicture} [scale = 0.35]
\foreach \x/\y in {0/0, 1/0, 2/0, 3/0, 4/0, 5/0 }
\filldraw[color = lg,opacity=.3] (\x, \y) rectangle (\x+1, \y+1);
\foreach \x/\y in { 2/0, 3/0, 4/0}
\filldraw[color = lightgreen,opacity=.4] (\x, \y) rectangle (\x+1, \y+1);
\draw[step=1cm,gray, thin] (0, 0) grid (6,1); 
\draw (0.5, .5) node {\small{$1$}};\draw (1.5, .5) node {\small{$1$}};\draw (2.5, .5) node {\small{$0$}};\draw (3.5, .5) node {\small{$0$}};\draw (4.5, .5) node {\small{$1$}};\draw (5.5, .5) node {\small{$1$}};
\end{tikzpicture}\hspace{.5 cm}
\begin{tikzpicture} [scale = 0.35]
\foreach \x/\y in {0/0, 1/0, 2/0, 3/0, 4/0, 5/0 }
\filldraw[color = lg,opacity=.3] (\x, \y) rectangle (\x+1, \y+1);
\foreach \x/\y in { 2/0, 3/0, 4/0}
\filldraw[color = lightgreen,opacity=.4] (\x, \y) rectangle (\x+1, \y+1);
\draw[step=1cm,gray, thin] (0, 0) grid (6,1); 
\draw (0.5, .5) node {\small{$1$}};\draw (1.5, .5) node {\small{$1$}};\draw (2.5, .5) node {\small{$0$}};\draw (3.5, .5) node {\small{$0$}};\draw (4.5, .5) node {\small{$0$}};\draw (5.5, .5) node {\small{$1$}};
\end{tikzpicture}
\vspace{0.1 cm}

\begin{tikzpicture} [scale = 0.35]
\foreach \x/\y in {0/0, 1/0, 2/0, 3/0, 4/0, 5/0 }
\filldraw[color = lg,opacity=.3] (\x, \y) rectangle (\x+1, \y+1);
\foreach \x/\y in { 2/0, 3/0, 4/0}
\filldraw[color = lightgreen,opacity=.4] (\x, \y) rectangle (\x+1, \y+1);
\draw[step=1cm,gray, thin] (0, 0) grid (6,1); 
\draw (4.5, -.5) node {\small{$x$}};
\draw (0.5, .5) node {\small{$0$}};\draw (1.5, .5) node {\small{$0$}};\draw (2.5, .5) node {\small{$1$}};\draw (3.5, .5) node {\small{$1$}};\draw (4.5, .5) node {\small{$1$}};\draw (5.5, .5) node {\small{$0$}};
\end{tikzpicture}\hspace{.5 cm}
\begin{tikzpicture} [scale = 0.35]
\foreach \x/\y in {0/0, 1/0, 2/0, 3/0, 4/0, 5/0 }
\filldraw[color = lg,opacity=.3] (\x, \y) rectangle (\x+1, \y+1);
\foreach \x/\y in { 2/0, 3/0, 4/0}
\filldraw[color = lightgreen,opacity=.4] (\x, \y) rectangle (\x+1, \y+1);
\draw[step=1cm,gray, thin] (0, 0) grid (6,1); 
\draw (4.5, -.5) node {\small{$x$}};
\draw (0.5, .5) node {\small{$0$}};\draw (1.5, .5) node {\small{$0$}};\draw (2.5, .5) node {\small{$0$}};\draw (3.5, .5) node {\small{$1$}};\draw (4.5, .5) node {\small{$0$}};\draw (5.5, .5) node {\small{$0$}};
\end{tikzpicture}\hspace{.5 cm}
\begin{tikzpicture} [scale = 0.35]
\foreach \x/\y in {0/0, 1/0, 2/0, 3/0, 4/0, 5/0 }
\filldraw[color = lg,opacity=.3] (\x, \y) rectangle (\x+1, \y+1);
\foreach \x/\y in { 2/0, 3/0, 4/0}
\filldraw[color = lightgreen,opacity=.4] (\x, \y) rectangle (\x+1, \y+1);
\draw[step=1cm,gray, thin] (0, 0) grid (6,1); 
\draw (4.5, -.5) node {\small{$x$}};
\draw (0.5, .5) node {\small{$1$}};\draw (1.5, .5) node {\small{$1$}};\draw (2.5, .5) node {\small{$1$}};\draw (3.5, .5) node {\small{$0$}};\draw (4.5, .5) node {\small{$0$}};\draw (5.5, .5) node {\small{$0$}};
\end{tikzpicture}
\vspace{1 mm}
\caption{The upper update windows give $\CA(x, t)=0$, whereas the lower ones give $\CA(x, t)=1$, if $B = 2$.}\label{fig:CAwindow}
\end{center}
\end{figure}
In the next section, concerning the 2-player game, we will see that the green parts in the update window, of size $\Gamma$, will decide the shape of the play triangles, and the full window of size $\Delta$ will determine the optimal play.

\section{The 2-player game}\label{sec:game}
There are four parameters deciding the setting of this game, $(\gamma, \ell, r, b)$ and a function $\delta=\gamma+\ell+r$, satisfying $\delta > b\ge 0$, $r\ge 0$, $\ell\ge 0$, $\gamma\ge 2$.
(In Section~\ref{sec:frac} we will require $\gamma =2 , b=0$.) A \emph{play-triangle} $T=T_{\gamma, l, r}(x,y,h)$ is the set 
\begin{align}\label{eq:T}
T=\bigcup_{i=1}^h\{(x - (i-1)(\gamma - 1), h-i+y), \ldots ,(x, h-i+y)\}.
\end{align}
If $h=1$ then the triangle is a single cell, and in general the top of the triangle is the point $\text{top}(T)=(x,y+h-1)$ and the base of $T$ is the set $\text{base}(T)=\{(x - (h-1)(\gamma - 1), y), \ldots ,(x, y)\}$.  The support of T is $\text{support}(T)=\{(x - h(\gamma - 1) - \ell, y-1), \ldots ,(x + r, y-1)\}$. The parameter $\ell$ counts the number of cells to the left of the extension of the play-triangle inside the support, and similarly $r$ counts the number of cells to the right; see Figure~\ref{fig:triangle1}. 
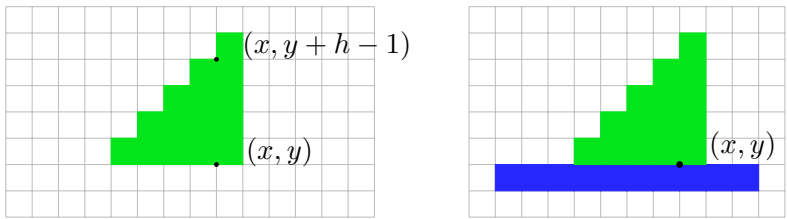
\begin{figure} [ht!]
\vspace{.5 cm}
\begin{center}
\begin{tikzpicture} [scale = 0.35]
\draw[step=1cm, lightgray] (-1, -1) grid (13,7); 
\foreach \x/\y in { 3/1, 4/1, 5/1,6/1, 7/1}
\filldraw[color = lightgreen,opacity=1] (\x, \y) rectangle (\x+1, \y+1);
\foreach \x/\y in {4/2, 5/2, 6/2,7/2 }
\filldraw[color = lightgreen,opacity=1] (\x, \y) rectangle (\x+1, \y+1);
\foreach \x/\y in {5/3, 6/3, 7/3 }
\filldraw[color = lightgreen,opacity=1] (\x, \y) rectangle (\x+1, \y+1);
\foreach \x/\y in {6/4,7/4 }
\filldraw[color = lightgreen,opacity=1] (\x, \y) rectangle (\x+1, \y+1);
\foreach \x/\y in {7/5 }
\filldraw[color = lightgreen,opacity=1] (\x, \y) rectangle (\x+1, \y+1);
\draw (9.4, 1.5) node {$(x,y)$};
\filldraw[color =black, opacity=1]  (7, 1) circle (2pt);
\draw (11.2, 5.5) node {$(x,y+h-1)$};
\filldraw[color =black, opacity=1]  (7, 5) circle (2pt);
\end{tikzpicture}\hspace{.5 cm}
\begin{tikzpicture} [scale = 0.35]
\draw[step=1cm, lightgray] (-1, -1) grid (11,7); 
\foreach \x/\y in {0/0, 1/0, 2/0, 3/0, 4/0, 5/0, 6/0,7/0,8/0,9/0 }
\filldraw[color = darkblue,opacity=1] (\x, \y) rectangle (\x+1, \y+1);
\foreach \x/\y in { 3/1, 4/1, 5/1,6/1, 7/1}
\filldraw[color = lightgreen,opacity=1] (\x, \y) rectangle (\x+1, \y+1);
\foreach \x/\y in {4/2, 5/2, 6/2,7/2 }
\filldraw[color = lightgreen,opacity=1] (\x, \y) rectangle (\x+1, \y+1);
\foreach \x/\y in {5/3, 6/3, 7/3 }
\filldraw[color = lightgreen,opacity=1] (\x, \y) rectangle (\x+1, \y+1);
\foreach \x/\y in {6/4,7/4 }
\filldraw[color = lightgreen,opacity=1] (\x, \y) rectangle (\x+1, \y+1);
\foreach \x/\y in {7/5 }
\filldraw[color = lightgreen,opacity=1] (\x, \y) rectangle (\x+1, \y+1);
\draw (9.4, 1.7) node {$(x,y)$};
\filldraw[color =black, opacity=1]  (7, 1) circle (3pt);
\end{tikzpicture}
\end{center}
\caption{To the left, a play triangle in green $T_{\gamma, l, r}(x,y,5)$, and to the right, the same triangle with its support in blue, with $\ell=r=2$.}\label{fig:triangle1}
\end{figure}

\subsection{Generic play: case IRT}

In this section $\gamma =2$, so we play \emph{isosceles right-angle triangles} (IRT) \cite{L13}, but the setting can be translated to arbitrary $\gamma$. 
The game starts with an arbitrary triangle position of the form $T=T(x,y,h)$, where $(x,y)$ is the location of the lower right cell in the triangle, and $h$ is the (number of cells in the) height of the triangle; see Figure~\ref{fig:triangle1} for points of reference. Hence, the position space is $\{(x,y,h)\mid x\in \mathbb Z, y\in \mathbb Z_{\ge 0}, h\in \mathbb Z_{>0}\}$. For a starting position, we assume that $y$ is large, but $x$ is arbitrary. The height $h$ of the triangle is also arbitrary, although it is convenient to start with a fairly small play-triangle.

At first, say, player $A$ proposes a `hopeful' \emph{play-window} of size $\delta$ intersecting the support of the current triangle; see the leftmost picture in Figure~\ref{fig:discussion}, where $\delta=4$. This is not yet a move; the purpose is to start a short `discussion' with player $B$ of the possible move options. Next, player $B$ blocks off at most $b$ of the cells in the window; in the middle picture of Figure~\ref{fig:discussion}, we have $b=2$. 

At this point, player $A$ chooses to play the top of the next triangle (of any size within the game board and not hitting any terminal obstacle). This top must be in one of the non-blocked (light) cells in the window. 

\begin{figure} [ht!]
\vspace{.5 cm}
\begin{center}
\begin{tikzpicture} [scale = 0.3]
\draw[step=1cm, lightgray] (-1, -3) grid (11,7); 

\foreach \x/\y in {0/0, 1/0, 2/0, 3/0, 4/0, 5/0, 6/0,7/0,8/0,9/0 }
\filldraw[color = darkblue,opacity=1] (\x, \y) rectangle (\x+1, \y+1);
\foreach \x/\y in {2/0, 3/0, 4/0, 5/0 }
\filldraw[color = lg,opacity=1] (\x, \y) rectangle (\x+1, \y+1);
\foreach \x/\y in { 3/1, 4/1, 5/1,6/1, 7/1}
\filldraw[color = lightgreen,opacity=1] (\x, \y) rectangle (\x+1, \y+1);
\foreach \x/\y in {4/2, 5/2, 6/2,7/2 }
\filldraw[color = lightgreen,opacity=1] (\x, \y) rectangle (\x+1, \y+1);
\foreach \x/\y in {5/3, 6/3, 7/3 }
\filldraw[color = lightgreen,opacity=1] (\x, \y) rectangle (\x+1, \y+1);
\foreach \x/\y in {6/4,7/4 }
\filldraw[color = lightgreen,opacity=1] (\x, \y) rectangle (\x+1, \y+1);
\foreach \x/\y in {7/5 }
\filldraw[color = lightgreen,opacity=1] (\x, \y) rectangle (\x+1, \y+1);
\draw[step=1cm,lightgray, thin] (2, 0) grid (6,1); 
\end{tikzpicture}\hspace{.5 cm}
\begin{tikzpicture} [scale = 0.3]
\draw[step=1cm, lightgray] (-1, -3) grid (11,7); 
\foreach \x/\y in {0/0, 1/0, 2/0, 3/0, 4/0, 5/0, 6/0,7/0,8/0,9/0 }
\filldraw[color = darkblue,opacity=1] (\x, \y) rectangle (\x+1, \y+1);
\foreach \x/\y in {2/0, 3/0, 4/0, 5/0 }
\filldraw[color = lg,opacity=1] (\x, \y) rectangle (\x+1, \y+1);
%%block%%%%%%%%%%%%%%%
\foreach \x/\y in {
3/0, 4/0
}
{
\filldraw[color = darkred, opacity=1]  (\x+.5, \y+.5) circle (12pt);
}
\foreach \x/\y in { 3/1, 4/1, 5/1,6/1, 7/1}
\filldraw[color = lightgreen,opacity=1] (\x, \y) rectangle (\x+1, \y+1);
\foreach \x/\y in {4/2, 5/2, 6/2,7/2 }
\filldraw[color = lightgreen,opacity=1] (\x, \y) rectangle (\x+1, \y+1);
\foreach \x/\y in {5/3, 6/3, 7/3 }
\filldraw[color = lightgreen,opacity=1] (\x, \y) rectangle (\x+1, \y+1);
\foreach \x/\y in {6/4,7/4 }
\filldraw[color = lightgreen,opacity=1] (\x, \y) rectangle (\x+1, \y+1);
\foreach \x/\y in {7/5 }
\filldraw[color = lightgreen,opacity=1] (\x, \y) rectangle (\x+1, \y+1);
\draw[step=1cm,lightgray, thin] (2, 0) grid (6,1); 
\end{tikzpicture}\hspace{.5 cm}
\begin{tikzpicture} [scale = 0.3]
\draw[step=1cm, lightgray] (-1, -3) grid (11,7); 
\foreach \x/\y in {0/0, 1/0, 2/0, 3/0, 4/0, 5/0, 6/0,7/0,8/0,9/0 }
\filldraw[color = darkblue, opacity=1] (\x, \y) rectangle (\x+1, \y+1);
\foreach \x/\y in {2/0, 3/0, 4/0, 5/0 }
\filldraw[color = lg,opacity=1] (\x, \y) rectangle (\x+1, \y+1);
\foreach \x/\y in { 3/1, 4/1, 5/1,6/1, 7/1}
\filldraw[color = lightgreen2, opacity=1] (\x, \y) rectangle (\x+1, \y+1);
\draw[step=1cm, lightgray, thin, dashed] (3, 1) grid (8,2); 
\foreach \x/\y in {4/2, 5/2, 6/2,7/2 }
\filldraw[color = lightgreen2, opacity=1] (\x, \y) rectangle (\x+1, \y+1);
\draw[step=1cm, lightgray, thin, dashed] (4, 2) grid (8,3); 
\foreach \x/\y in {5/3, 6/3, 7/3 }
\filldraw[color = lightgreen2, opacity=1] (\x, \y) rectangle (\x+1, \y+1);
\draw[step=1cm, lightgray, thin, dashed] (5,3) grid (8,4); 
\foreach \x/\y in {6/4,7/4 }
\filldraw[color = lightgreen2, opacity=1] (\x, \y) rectangle (\x+1, \y+1);
\draw[step=1cm, lightgray, thin, dashed] (6,4) grid (8,5); 
\foreach \x/\y in {7/5 }
\filldraw[color = lightgreen2, opacity=1] (\x, \y) rectangle (\x+1, \y+1);
\draw[step=1cm, lightgray, thin, dashed] (7,5) grid (8,6); 
%%block%%%%%%%%%%%%%%%
\foreach \x/\y in {
3/0,4/0
}
{
\filldraw[color = darkred, opacity=1]  (\x+.5, \y+.5) circle (12pt);
}
%%%%%%%%%%next move%%%%%%%%%%%%%%%
\foreach \x/\y in {3/-2, 4/-2, 5/-2 }
\filldraw[color = lightgreen, opacity=1] (\x, \y) rectangle (\x+1, \y+1);
\foreach \x/\y in {4/-1,5/-1 }
\filldraw[color = lightgreen, opacity=1] (\x, \y) rectangle (\x+1, \y+1);
\foreach \x/\y in {5/0}
\filldraw[color = lightgreen, opacity=1] (\x, \y) rectangle (\x+1, \y+1);
\draw[step=1cm,lightgray, thin] (2, 0) grid (5,1); 
\end{tikzpicture}
\end{center}
\caption{The move discussion. (For references to color, see online version.)}\label{fig:discussion}
\end{figure}
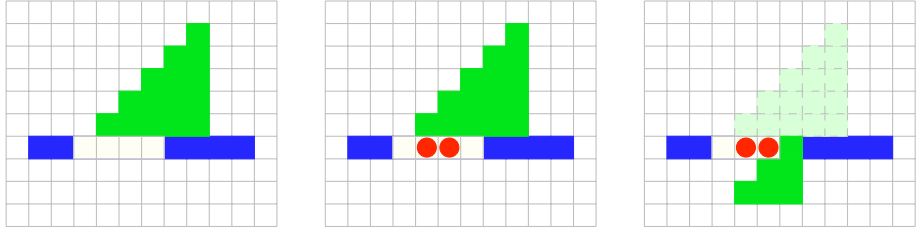

Each move discussion is particular to each stage of game. For each move there will be a new discussion (by alternating the roles of the players). 
The play proceeds exactly in the same manner, with the player alternating turns, until the base of a triangle approaches the terminal level at $y=0$; or, as we will see, possibly already at $y=1$, depending on the locations of the obstacles. 

\subsection{The final stage of game}
The set of terminal level \emph{obstacles} $\Omega \subset \{ (x,0)\mid x\in \mathbb{Z}\} \neq \varnothing$ is announced before the game starts. The obstacles affect the placement of a triangle, but not its support. The terminal play is enhanced by the rule that no triangle can intersect an obstacle, whereas its support easily surrounds the obstacles. 

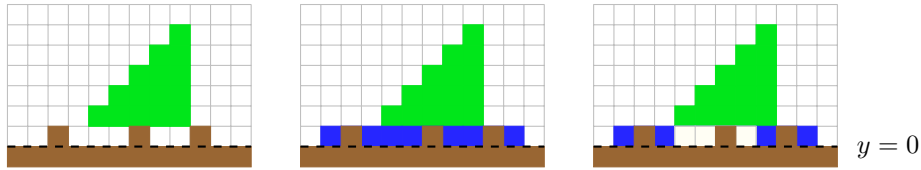
\begin{figure} [ht!]
%\vspace{.3 cm}
\begin{center}

\begin{tikzpicture} [scale = 0.27]
\draw[step=1cm, lightgray] (-1, 0) grid (11,7); 
\foreach \x/\y in { 3/1, 4/1, 5/1,6/1, 7/1}
\filldraw[color = lightgreen,opacity=1] (\x, \y) rectangle (\x+1, \y+1);
\foreach \x/\y in {4/2, 5/2, 6/2,7/2 }
\filldraw[color = lightgreen,opacity=1] (\x, \y) rectangle (\x+1, \y+1);
\foreach \x/\y in {5/3, 6/3, 7/3 }
\filldraw[color = lightgreen,opacity=1] (\x, \y) rectangle (\x+1, \y+1);
\foreach \x/\y in {6/4,7/4 }
\filldraw[color = lightgreen,opacity=1] (\x, \y) rectangle (\x+1, \y+1);
\foreach \x/\y in {7/5 }
\filldraw[color = lightgreen,opacity=1] (\x, \y) rectangle (\x+1, \y+1);
\filldraw[color =darkbrown, opacity=1]  (1, 0) rectangle (2,1);
\filldraw[color =darkbrown, opacity=1]  (5, 0) rectangle (6,1);
\filldraw[color =darkbrown, opacity=1]  (8, 0) rectangle (9,1);
\filldraw[color =darkbrown, opacity=1]  (-1, -1) rectangle (11,0);
\draw[thick, dashed] (-1,0) -- (11,0);
\end{tikzpicture}\hspace{.5 cm}
%%%%%%%%%%%%%%%%%%%%%%%%%% NEWPIC %%%%%%%%%%%%%%%%%%%%%
\begin{tikzpicture} [scale = 0.27]
\draw[step=1cm, lightgray] (-1, 0) grid (11,7); 
\foreach \x/\y in { 3/1, 4/1, 5/1,6/1, 7/1}
\filldraw[color = lightgreen,opacity=1] (\x, \y) rectangle (\x+1, \y+1);
\foreach \x/\y in {4/2, 5/2, 6/2,7/2 }
\filldraw[color = lightgreen,opacity=1] (\x, \y) rectangle (\x+1, \y+1);
\foreach \x/\y in {5/3, 6/3, 7/3 }
\filldraw[color = lightgreen,opacity=1] (\x, \y) rectangle (\x+1, \y+1);
\foreach \x/\y in {6/4,7/4 }
\filldraw[color = lightgreen,opacity=1] (\x, \y) rectangle (\x+1, \y+1);
\foreach \x/\y in {7/5 }
\filldraw[color = lightgreen,opacity=1] (\x, \y) rectangle (\x+1, \y+1);
\foreach \x/\y in {0/0, 1/0, 2/0, 3/0, 4/0, 5/0, 6/0,7/0,8/0,9/0 }
\filldraw[color = darkblue,opacity=1] (\x, \y) rectangle (\x+1, \y+1);
\filldraw[color =darkbrown, opacity=1]  (1, 0) rectangle (2,1);
\filldraw[color =darkbrown, opacity=1]  (5, 0) rectangle (6,1);
\filldraw[color =darkbrown, opacity=1]  (8, 0) rectangle (9,1);
\filldraw[color =darkbrown, opacity=1]  (-1, -1) rectangle (11,0);
\draw[thick, dashed] (-1,0) -- (11,0);
\end{tikzpicture}\hspace{.5 cm}
%%%%%%%%%%%%%%%%%%%%%%%%%% NEWPIC %%%%%%%%%%%%%%%%%%%%%
\begin{tikzpicture} [scale = 0.27]
\draw[step=1cm, lightgray] (-1, 0) grid (11,7); 
\foreach \x/\y in { 3/1, 4/1, 5/1,6/1, 7/1}
\filldraw[color = lightgreen,opacity=1] (\x, \y) rectangle (\x+1, \y+1);
\foreach \x/\y in {4/2, 5/2, 6/2,7/2 }
\filldraw[color = lightgreen,opacity=1] (\x, \y) rectangle (\x+1, \y+1);
\foreach \x/\y in {5/3, 6/3, 7/3 }
\filldraw[color = lightgreen,opacity=1] (\x, \y) rectangle (\x+1, \y+1);
\foreach \x/\y in {6/4,7/4 }
\filldraw[color = lightgreen,opacity=1] (\x, \y) rectangle (\x+1, \y+1);
\foreach \x/\y in {7/5 }
\filldraw[color = lightgreen,opacity=1] (\x, \y) rectangle (\x+1, \y+1);
\draw (13.5, 0) node {\small{$y=0$}};
\foreach \x/\y in {0/0, 1/0, 2/0, 3/0, 4/0, 5/0, 6/0,7/0,8/0,9/0 }
\filldraw[color = darkblue,opacity=1] (\x, \y) rectangle (\x+1, \y+1);
\foreach \x/\y in {3/0, 4/0, 5/0, 6/0 }
\filldraw[color = lg,opacity=1] (\x, \y) rectangle (\x+1, \y+1);
\draw[step=1cm, lightgray] (3, 0) grid (7,1); 
\filldraw[color =darkbrown, opacity=1]  (1, 0) rectangle (2,1);
\filldraw[color =darkbrown, opacity=1]  (5, 0) rectangle (6,1);
\filldraw[color =darkbrown, opacity=1]  (8, 0) rectangle (9,1);
\filldraw[color =darkbrown, opacity=1]  (-1, -1) rectangle (11,0);
\draw[thick, dashed] (-1,0) -- (11,0);
\end{tikzpicture}
\end{center}
\caption{When play approaches the terminal level.}\label{fig:terminal}
\end{figure}
%%%%%%%%%%%%%%%%%% blocking terminal %%%%%%%%%%%%%%%%%%%%%%%%
\begin{figure} [ht!]
\begin{center}
\begin{tikzpicture} [scale = 0.3]
\draw[step=1cm, lightgray] (-1, 0) grid (11,7); 
\foreach \x/\y in { 3/1, 4/1, 5/1,6/1, 7/1}
\filldraw[color = lightgreen,opacity=1] (\x, \y) rectangle (\x+1, \y+1);
\foreach \x/\y in {4/2, 5/2, 6/2,7/2 }
\filldraw[color = lightgreen,opacity=1] (\x, \y) rectangle (\x+1, \y+1);
\foreach \x/\y in {5/3, 6/3, 7/3 }
\filldraw[color = lightgreen,opacity=1] (\x, \y) rectangle (\x+1, \y+1);
\foreach \x/\y in {6/4,7/4 }
\filldraw[color = lightgreen,opacity=1] (\x, \y) rectangle (\x+1, \y+1);
\foreach \x/\y in {7/5 }
\filldraw[color = lightgreen,opacity=1] (\x, \y) rectangle (\x+1, \y+1);
\foreach \x/\y in {0/0, 1/0, 2/0, 3/0, 4/0, 5/0, 6/0,7/0,8/0,9/0 }
\filldraw[color = darkblue,opacity=1] (\x, \y) rectangle (\x+1, \y+1);
\foreach \x/\y in {3/0, 4/0, 5/0, 6/0 }
\filldraw[color = lg,opacity=1] (\x, \y) rectangle (\x+1, \y+1);
\draw[step=1cm, lightgray] (3, 0) grid (7,1); 
\filldraw[color =darkbrown, opacity=1]  (1, 0) rectangle (2,1);
\filldraw[color =darkbrown, opacity=1]  (5, 0) rectangle (6,1);
\filldraw[color =darkbrown, opacity=1]  (8, 0) rectangle (9,1);
\filldraw[color =darkbrown, opacity=1]  (-1, -1) rectangle (11,0);
\foreach \x/\y in {
3/0, 4/0
}
{
\filldraw[color = darkred, opacity=1]  (\x+.5, \y+.5) circle (12pt);
}
\draw[thick, dashed] (-1,0) -- (11,0);
\end{tikzpicture}\hspace{1 cm}
%%%%%%%%%%%%%%%%%%%%%%%%%% NEWPIC %%%%%%%%%%%%%%%%%%%%%
\begin{tikzpicture} [scale = 0.3]
\draw[step=1cm, lightgray] (-1, 0) grid (11,7); 
\foreach \x/\y in { 3/1, 4/1, 5/1,6/1, 7/1}
\filldraw[color = lightgreen,opacity=1] (\x, \y) rectangle (\x+1, \y+1);
\foreach \x/\y in {4/2, 5/2, 6/2,7/2 }
\filldraw[color = lightgreen,opacity=1] (\x, \y) rectangle (\x+1, \y+1);
\foreach \x/\y in {5/3, 6/3, 7/3 }
\filldraw[color = lightgreen,opacity=1] (\x, \y) rectangle (\x+1, \y+1);
\foreach \x/\y in {6/4,7/4 }
\filldraw[color = lightgreen,opacity=1] (\x, \y) rectangle (\x+1, \y+1);
\foreach \x/\y in {7/5 }
\filldraw[color = lightgreen,opacity=1] (\x, \y) rectangle (\x+1, \y+1);
\draw (13.5, 0) node {\small{$y=0$}};
\foreach \x/\y in {0/0, 1/0, 2/0, 3/0, 4/0, 5/0, 6/0,7/0,8/0,9/0 }
\filldraw[color = darkblue,opacity=1] (\x, \y) rectangle (\x+1, \y+1);
\foreach \x/\y in {3/0, 4/0, 5/0, 6/0 }
\filldraw[color = lg,opacity=1] (\x, \y) rectangle (\x+1, \y+1);
\draw[step=1cm, lightgray] (3, 0) grid (7,1); 

\filldraw[color =darkbrown, opacity=1]  (1, 0) rectangle (2,1);
\filldraw[color =darkbrown, opacity=1]  (5, 0) rectangle (6,1);
\filldraw[color =darkbrown, opacity=1]  (8, 0) rectangle (9,1);
\filldraw[color =darkbrown, opacity=1]  (-1, -1) rectangle (11,0);

\foreach \x/\y in {
3/0, 4/0,6/0
}
{
\filldraw[color = darkred, opacity=1]  (\x+.5, \y+.5) circle (12pt);
}
\draw[thick, dashed] (-1,0) -- (11,0);
\end{tikzpicture}
\end{center}
\caption{Blocking maneuvers and final play.}\label{fig:terminalblock}
\end{figure}
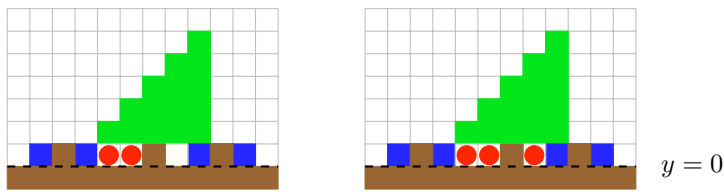
Consider Figure~\ref{fig:terminal} ($x=0$ is to the left in the diagram). To the left, we find a play triangle in green approaching the terminal level (with the base at $y=1$) and we display the relevant neighboorhood. There are obstacles $\{(2,0),(6,0),(9,0)\} = \Omega$. In the middle, we find the same triangle with its support in blue, with $\ell=r=2$. To the right, we show the chosen window in white. 

Consider Figure~\ref{fig:terminalblock}. To the left we show that for a blocking number of $b=2$, then the next player has a winning move, to place a single cell play-triangle in one of the white cells (all cannot be blocked). On the other hand, consider the blocking number $b=3$ (right). Then, for all placements of the window, each non-obstructed cell in the window can be blocked off. 

Observe the difference between the blocking maneuver of the previous player, and the fixed obstacles in $\Omega$. The next player wins on the penultimate level if and only if she finds a placement of the window such that $|W\cap \Omega | + b < \delta$, where $W$ is the set of cells in the play-window.

This discussion leads us to the next section, where we identify the obstacles of the terminal level of the 2-player game with the `1's of the CA's initial string, and show that optimal play in a game simply corresponds to the update of a CA (by identifying the parameters). 

\section{The game--CA correspondence}\label{sec: correspondence}

In this section we demonstrate that the CA and the outcomes of the games are equivalent. Consider $\ell =L, r=R, \gamma =\Gamma, b=B, t=y$. A play-triangle $T$ is \emph{CA-safe} if each underlying CA-cell is a `0', and in $\support (T)$, each underlying CA window $w_1$ contains at most $B$ `0's, i.e. $(x,y)\in T\Rightarrow \CA(x,y)=0 $ and $(x,y)\in \base(T)\Rightarrow |w_1|_0\le B$. This includes the case of the support being below the game board (at $y=-1$); each terminal triangle is CA-safe. The following lemma provides a connection between CAs and 2-player games. 
\begin{lemma}\label{lem:CAsafe}
Let $u, v\in \mathbb Z_{\ge 0}$. Then $\CA(u,v) = 0$ if and only if $\exists h\ge 1: T = T(u, v-h+1, h)$ is CA-safe.
\end{lemma}
\begin{proof}
Suppose $\CA(u,v) = 0$. Then, for all $h$, $\CA(\top(T)) = 0$. Hence, by the update rules of the CA, the case $|w_0|_1 = 0$, there must be a largest $h$ such that $(x,y)\in \base (T)$ implies $\CA(x,y)=0$. Thus, for all such $(x,y)$, $|w_1(x,y)|_0\le b$ (including the possibility of terminal $T$, i.e. $y=0$). The other direction is immediate by definition.
\end{proof}
\begin{theorem}\label{thm:1}
Let $\ell =L, r=R, \gamma =\Gamma, b'=B, t=y$. Then the previous player wins from the triangle-position $T=T(x,y,h)$ if and only if $T$ is CA-safe. 
\end{theorem}
\begin{proof}
We begin by proving that it is impossible to move from a CA-safe triangle $T_1$, to another CA-safe triangle $T_2$. In each window, in $\support (T_1)$, we find at most $B$ `0's among the CA-cells.  Therefore each `0' can be blocked off by the previous player. Therefore, the first player has to play $T_2$ such that $\CA(\top (T_2)) = 1$. Suppose that $(x, y)\in \base (T_2)$ implies $\CA(x, y) = 0$, which is a necessary condition for a CA-safe triangle. Then, by the update rule of the CA, as described in Lemma~\ref{lem:CAsafe}, we get $\top(T_2)=0$. Hence the first requirement of a CA-safe triangle is violated, and so $T_2$ is not CA-safe. 

Suppose next that $T_1$ is not CA-safe. Then one of the two conditions is violated. If $\exists (u,v)\in \base (T_2)$ such that $\CA(u,v) = 1$, then, by the update rule of the CA, $|w_1(u,v)|_0 > B$. Hence, by  the blocking rule, the first player can find $T_2$ such that $\CA(\top (T_2)) = 0$. By the update rule of the CA (Lemma~\ref{lem:CAsafe}), the first player finds a $T_2$ such that $(u,v)\in \base(T_2)$ implies $|w_1(u,v)|_0\le B$ (perhaps because $T_2$ is terminal). If $\forall (u,v)\in \base (T_1)$, $\CA(u,v) = 0$, then since $T_1$ is not CA-safe, $\support (T_1)$ cannot satisfy the given condition for a CA-safe triangle. Hence, the $w_0$-extensions in $\support(T_1)$ contains only ``0''s (for otherwise the condition for the $T_1$-base would be false). Since, for all triangles $T$,  $|\support (T) | > B$, in particular for this case, the second player cannot block off each `0' (in any window). Hence, by Lemma~\ref{lem:CAsafe}, the first player can play a CA-safe triangle $T_2$.
\end{proof}

\section{Fractals and partial convergence in games and CA}\label{sec:frac}
In this section we study sequences of CA, with the initial $\CA_{L,R}=\CA_{\Gamma, L,R, B}=\CA_{2, L, R, 0}$ (that is $B=0$ and  $\Gamma=2$), for some given $L$ and $R$. Our construction generalizes the classical self-similarity (Pascal's triangle modulo 2) in rule 60  ($L=R=0$), but here tending towards more complex fractals (Figure~\ref{fig:limgame}). We do not have the `zoom-in' similarity for a particular picture (like rule 60), but we rather obtain self-similarity in iterating the diagrams with a prescribed scaling factor of 2.

\begin{figure}[ht!]
\begin{center}
%\vspace{0.2 cm}
{\includegraphics[width=0.24\textwidth]{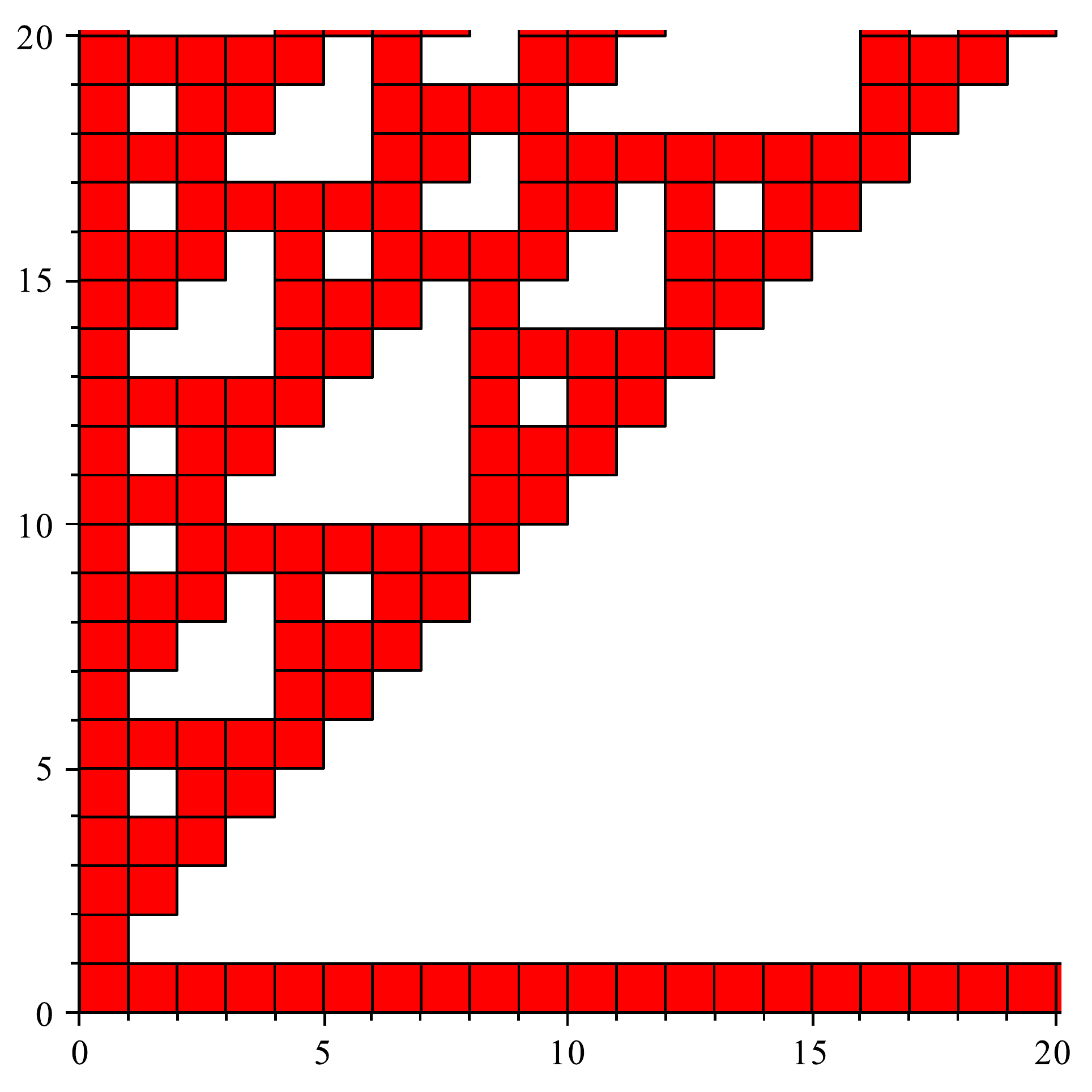}}\hspace{.1 cm}{\includegraphics[width=0.24\textwidth]{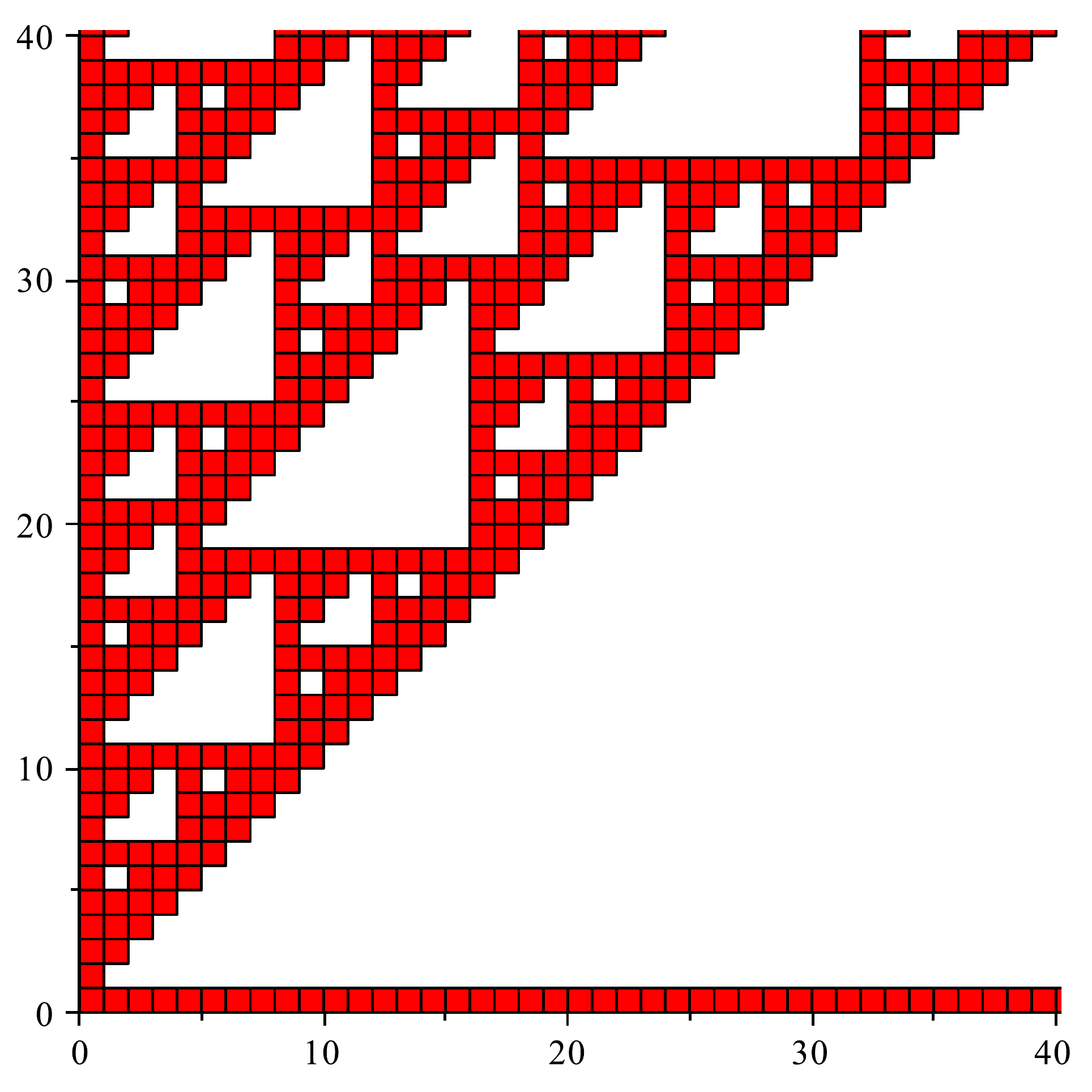}}
{\includegraphics[width=0.24\textwidth]{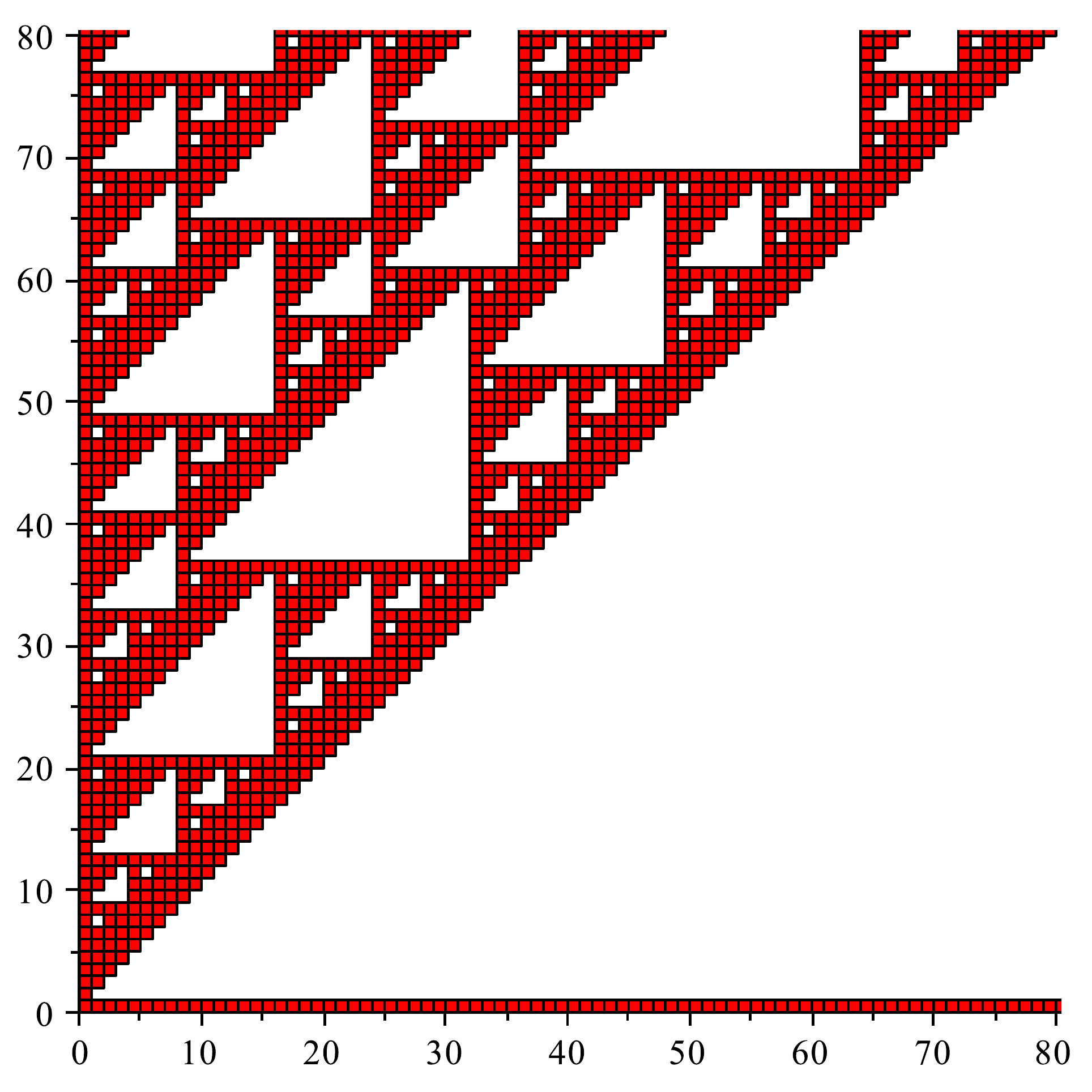}}\hspace{.1 cm}{\includegraphics[width=0.245\textwidth]{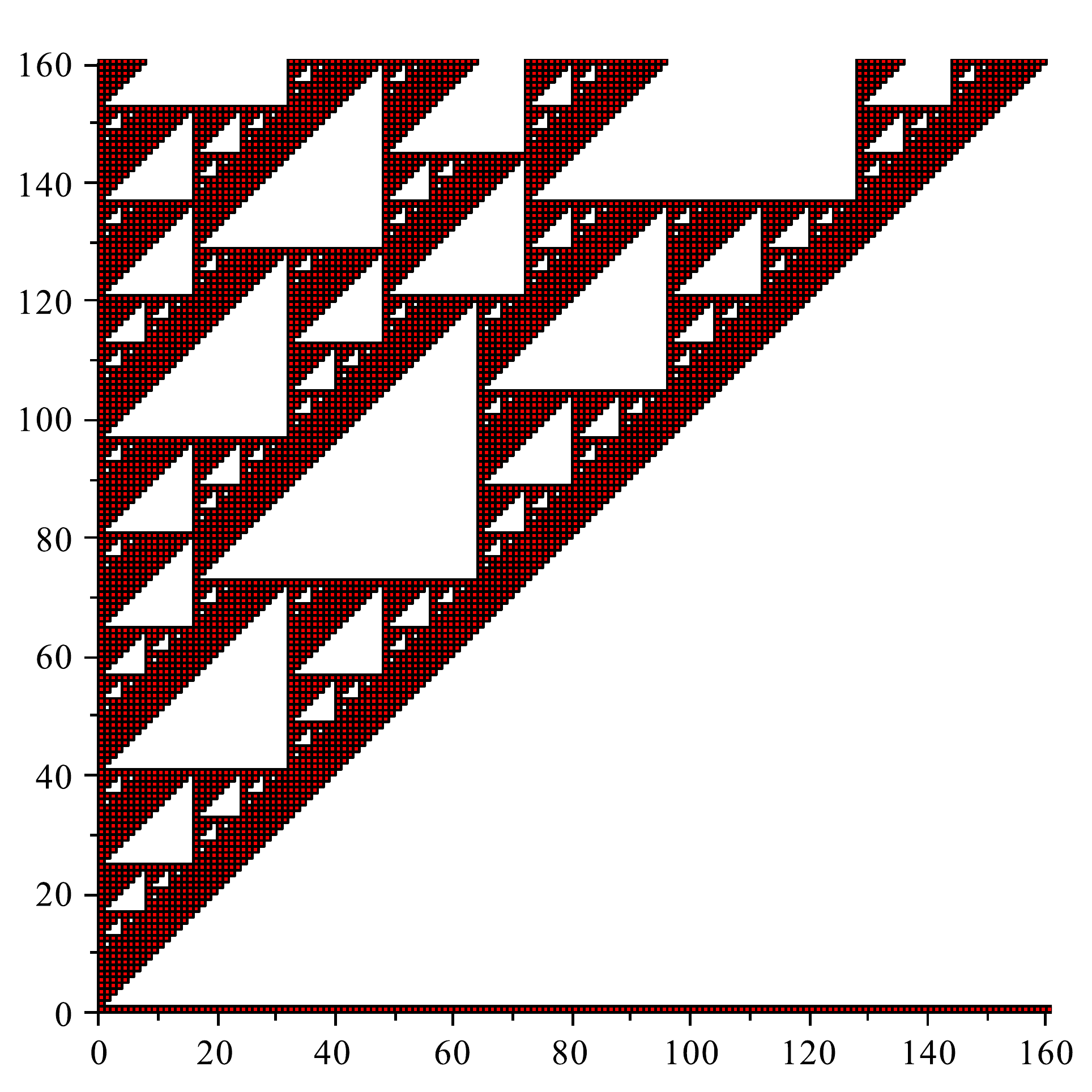}}
{\includegraphics[width=0.24\textwidth]{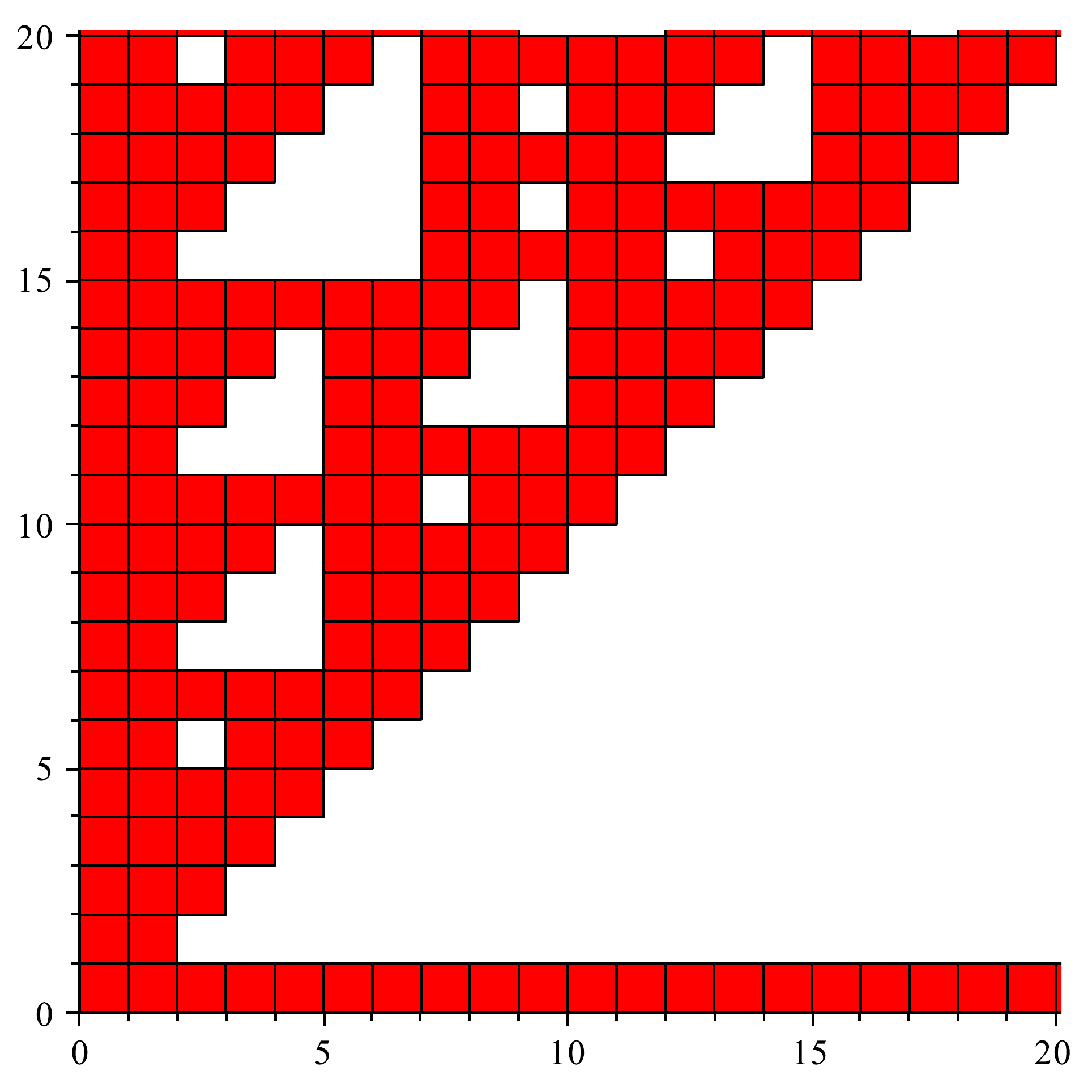}}\hspace{.1 cm}{\includegraphics[width=0.24\textwidth]{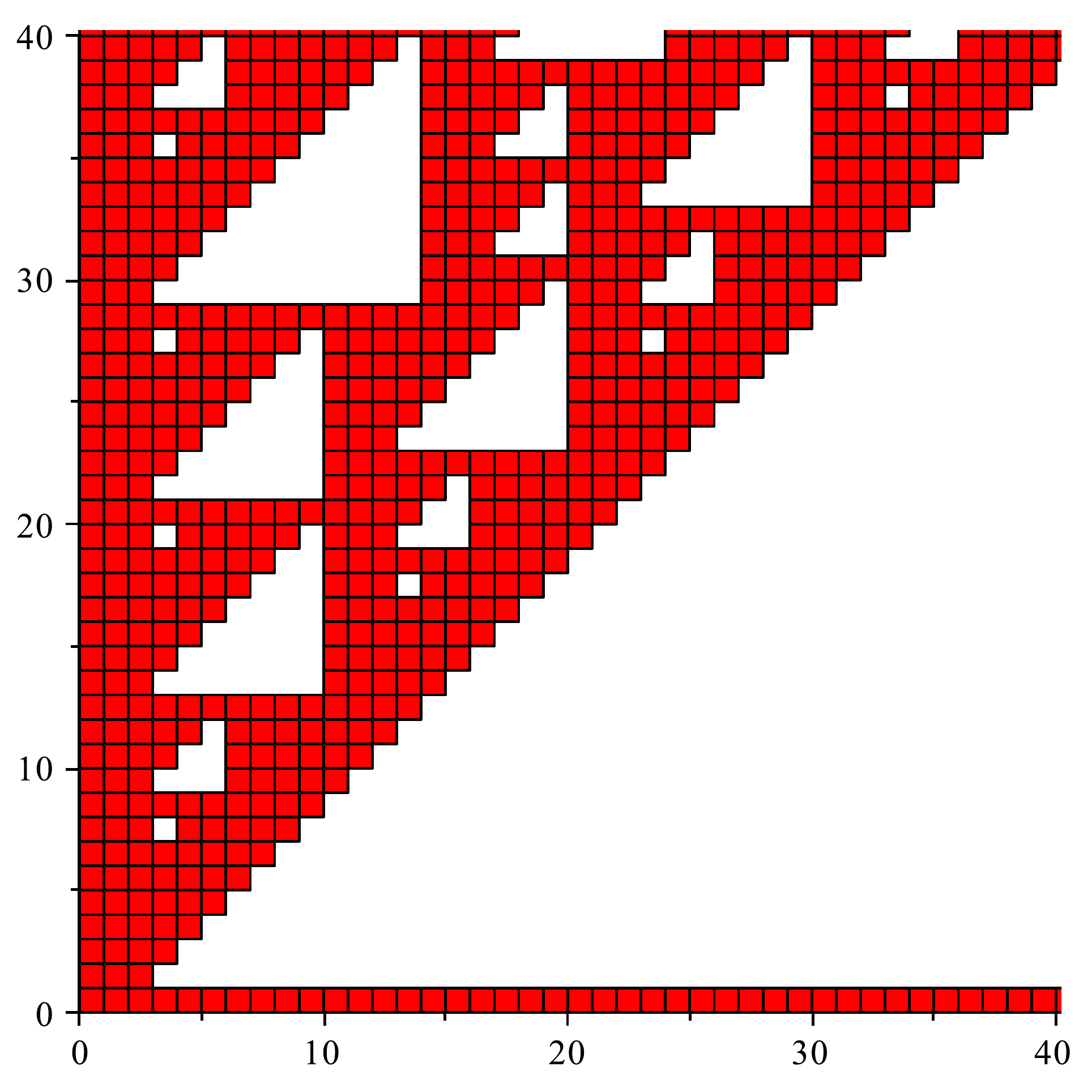}}
{\includegraphics[width=0.24\textwidth]{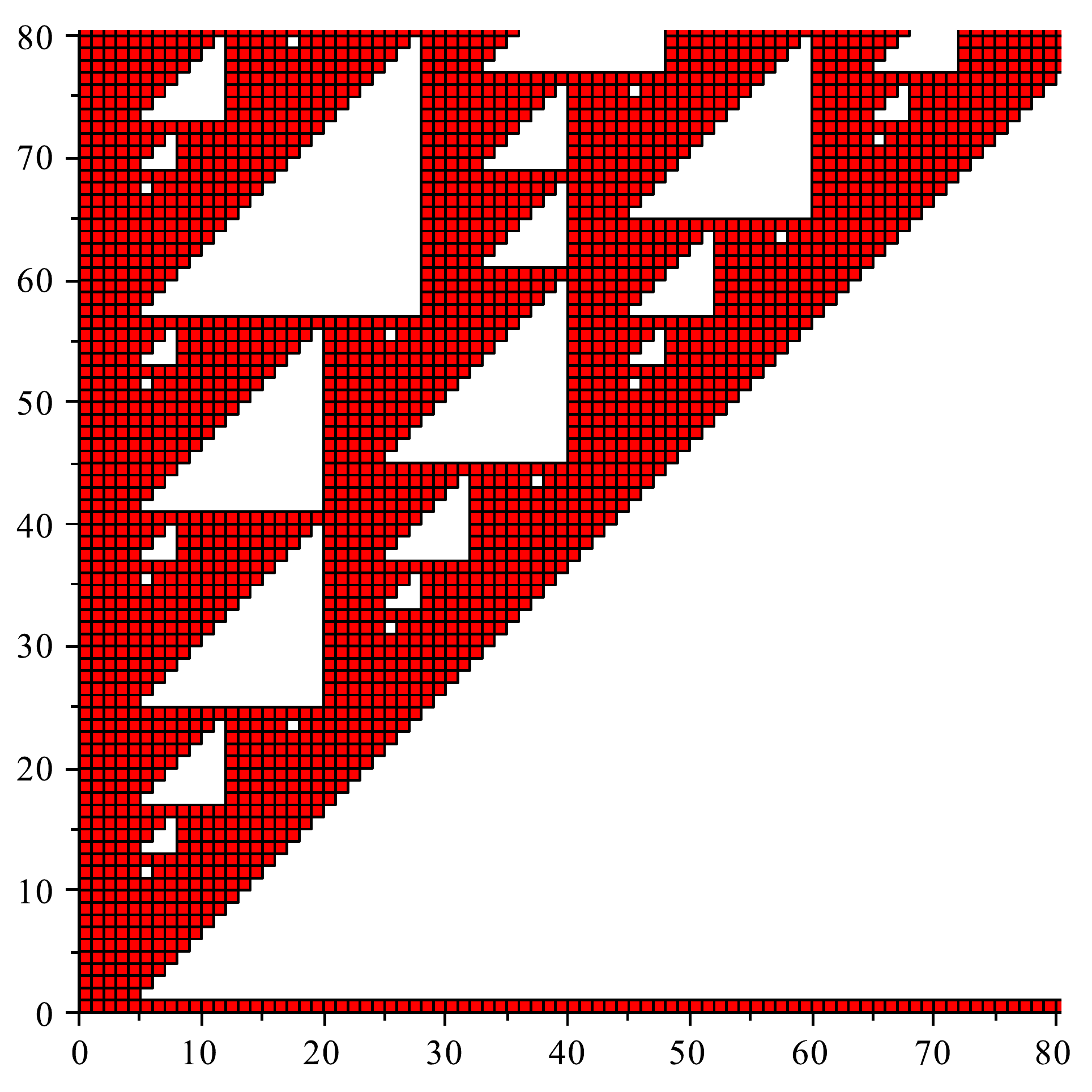}}\hspace{.1 cm}{\includegraphics[width=0.245\textwidth]{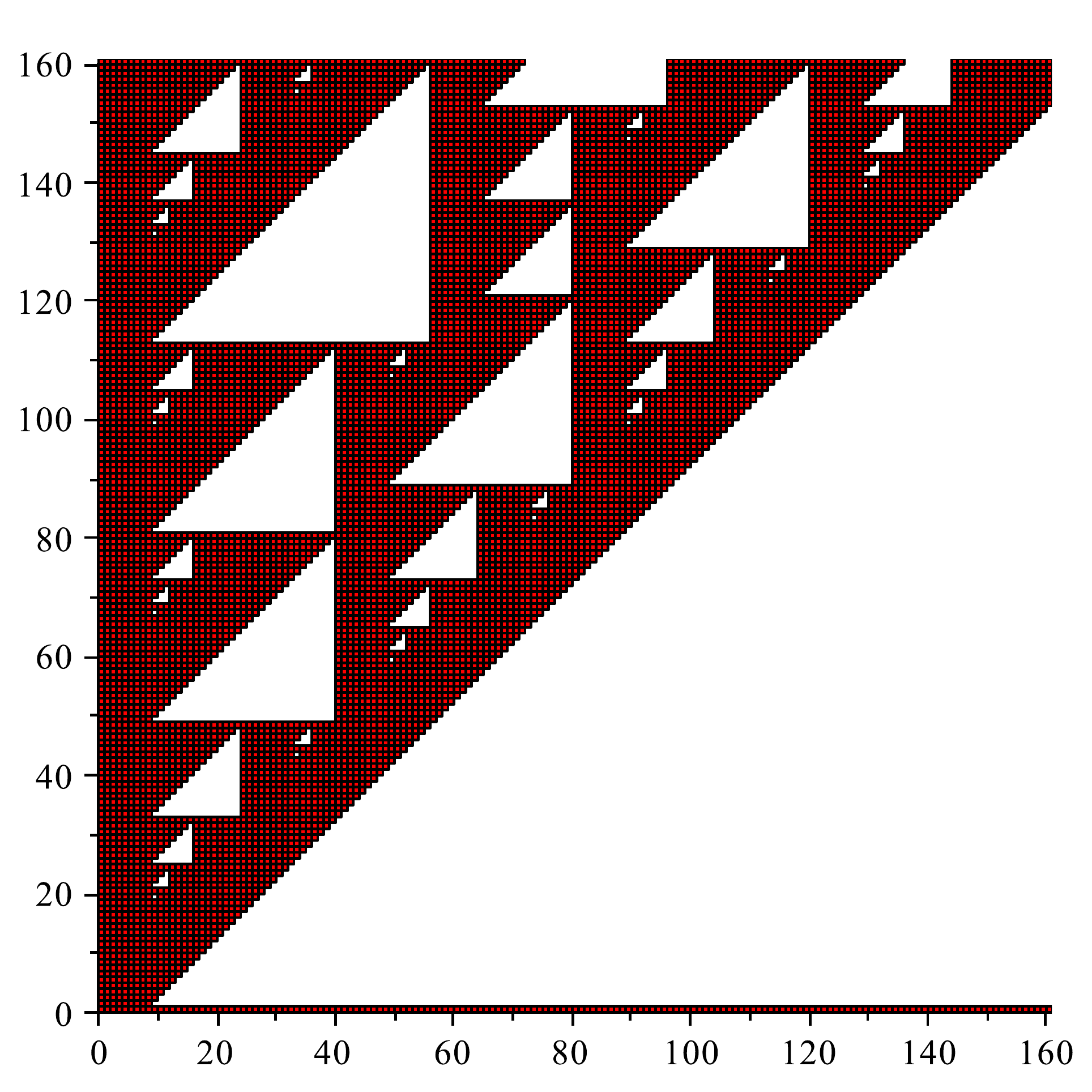}}
\end{center}\caption{The top figures display $CA_{L,R}$ for $(L,R)=(0,1),(0,2),(0,4),(0,8)$ and those below include $(L,R)=(1,2),(2,4),(4,8),(8,16)$; $\CA_{L,R}=(x, 0) = 1$ if and only if $x\geqslant 1$.}\label{fig:limgame}
\end{figure}

\begin{figure}[ht!]
\begin{center}
%\vspace{0.2 cm}
{\includegraphics[width=0.45\textwidth]{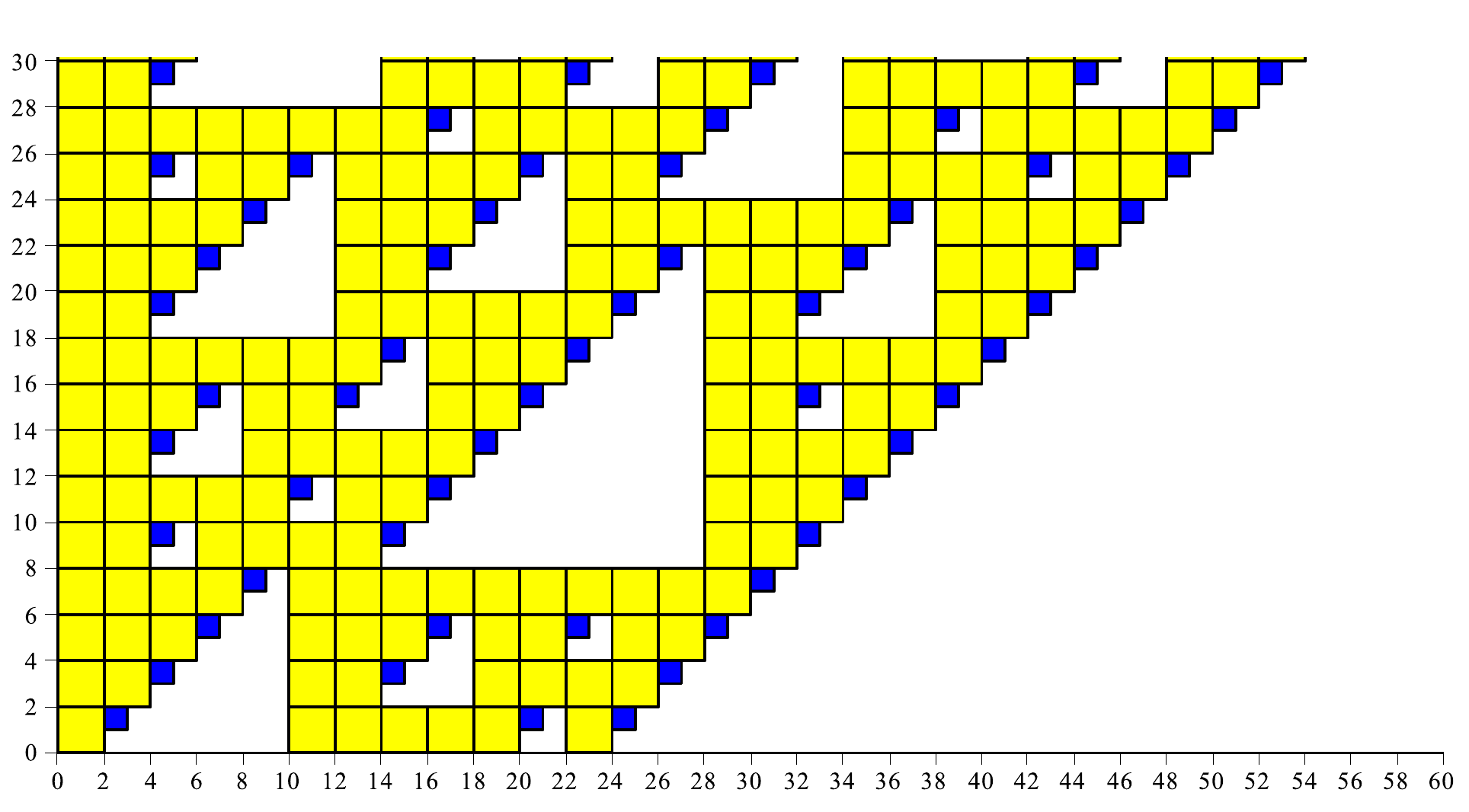}}
{\includegraphics[width=0.45\textwidth]{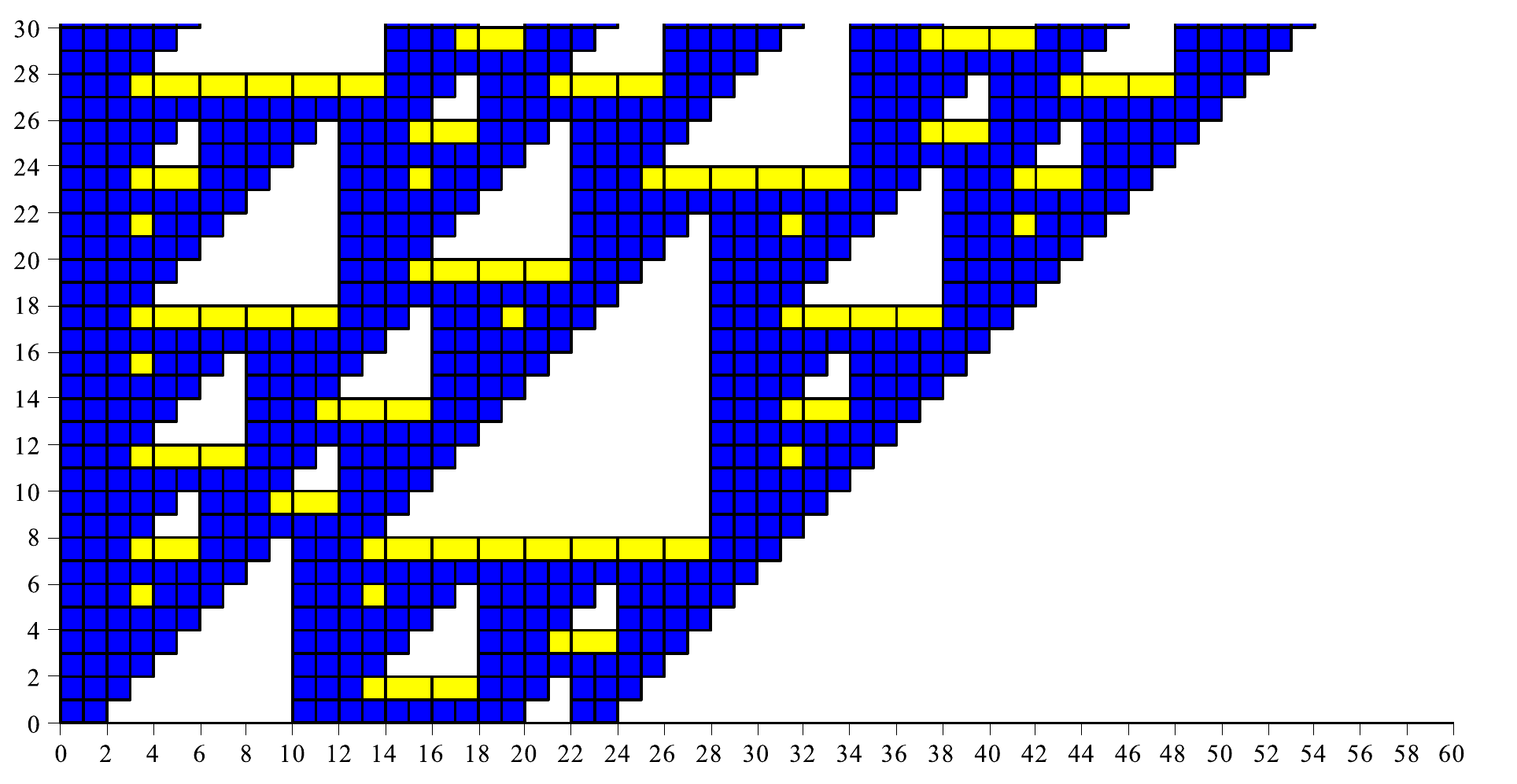}}
\end{center}\caption{Superposition of $\text{CA}_{1,1}$ (scaled yellow upper layer) $I^0(x)=1$ if $x=0, 5,6,7,8,9,11$, and $\text{CA}_{2,2}$ (blue upper layer) $I^1(x)=1$ if $x=0,1,5,\ldots, 19, 22,23$,  by (\ref{eq:I}).}\label{fig:limgame2}
\end{figure}

Let $I^0\in {\{0,1\}^\mathbb Z}$, and for $n \ge 0$ define a sequence of bit strings by 
\begin{align}\label{eq:I}
I^{n+1}(2x) = I^{n+1}(2x+1)=I^{n}(x).
\end{align}
For $n\in \mathbb Z_{\ge 0}$, define $\CA^n = \CA_{(2^nL,2^nR)}$ by the initial configurations $\CA^n(\cdot , 0) = I^n$; exemplified in Figure~\ref{fig:limgame2}. Note the case, for all $n$, $I^n(x) = 1$ if and only if $x\ge 0$;  Figure~\ref{fig:limgame}.
\begin{theorem}\label{thm:2} 
Let $x\in \mathbb Z, y\in \mathbb Z_{\ge 0}$ and $h\in \mathbb Z_{>0}$ and consider IRT play-triangles $T_1=T_1(u,v,h)$ and $T_2$ with $\base (T_1) = \{(u-h+1, v),\ldots , (u, v)\}$ and, in case $v=0$,  $\base(T_2) = \{(2(u-h)+2, 0),\ldots , (2u, 0)\}$, and otherwise $\base(T_2) = \{(2(u - h) + 1, 2v - 1),\ldots , (2u, 2v-1)\}$. 
Then $T_1$ is $\CA_{L , R}$-safe if and only if $T_2$ is $\CA_{2L, 2R}$-safe.
\end{theorem}

\begin{proof}
By definition (\ref{eq:I}) if $v = 0$ then $T_2$ is CA-safe if and only if $T_1$ is also. By induction, assume that the statement holds for all $ v < \mu$, and we prove that $T_1(u,\mu , h)$ is $\CA_{L , R}$-safe if and only if $T_2(2u, 2\mu -1, 2h)$ is $\CA_{2L , 2R}$-safe. Note that $\support (T_2) = \{(2(u - h) -2L, 2v - 2),\ldots , (2u+2R, 2v-2)\}$. The \emph{newborn stars} (single `0' cells in the second diagram) are special, because their existance is not revealed by the statement of the theorem, so they are not automatically assumed by induction; therefore their existence below level $\mu$ will be motivated given the other structure (triangels A, C and D in Figure~\ref{fig:proof1}), by induction.\\ 

\noindent Claim 1: The \emph{subsupport} (the cells just below the support) of $T_1$ contains a sequence of exactly $L + R + 1 = \Delta(L,R) - 1$ consecutive `1's i.e. $\exists \alpha , \beta : \alpha < x<\beta \Rightarrow\CA_{L, R}(x, \mu - 2) = 1 $ with $\beta -\alpha = \Delta(L,R) +1$, and $\CA_{L, R}(\alpha,\mu -2)=\CA_{L, R}(\beta,\mu -2) = 0$ if and only if there is a local single `0' in the subsupport of $T_2$, precisely $\CA_{2L, 2R}(2\alpha + 2L +4, 2\mu - 3) = 0$.\\

\noindent Proof of Claim 1: 
By induction, we get $2(L+R+1)+1=2L+2R+3=\Delta (2L,2R)+1$ consecutive `1's in $\subsupport (T_2) = \{(2(u - h) -2L, 2\mu - 3),\ldots , (2u+2R, 2\mu -3)\}$ (see Figure~\ref{fig:triangle1} dashed area and triangle C) and on the level just below, exactly $\Delta (2L,2R)$ consecutive cells, and the location of the single `0' follows.\\ 
%\end{proof}%of claim

 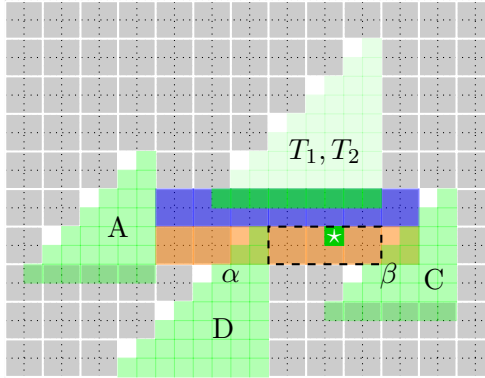
\begin{figure} [ht!]
\vspace{-.2 cm}
\begin{center}
\begin{tikzpicture} [scale = 0.5]
\filldraw[color = gray, opacity=.4] (4,6) rectangle (17,16);
\draw[step=.5cm, thin, dotted] (4, 6) grid (17,16);
\draw[step=1cm, thick, white] (4, 6) grid (17,16);
\begin{scope}[shift={(8,5)}]
\foreach \x/\y in {5/3, 6/3, 7/3 }
\filldraw[color = white,opacity=1] (\x, \y) rectangle (\x+1, \y+1);
\foreach \x/\y in {6/4,7/4 }
\filldraw[color = white,opacity=1] (\x, \y) rectangle (\x+1, \y+1);
\foreach \x/\y in {7/5 }
\filldraw[color = white,opacity=1] (\x, \y) rectangle (\x+1, \y+1);
\end{scope}
\begin{scope}[shift={(0,6)}]
\foreach \x/\y in {5/3, 6/3, 7/3 }
\filldraw[color = white,opacity=1] (\x, \y) rectangle (\x+1, \y+1);
\foreach \x/\y in {6/4,7/4 }
\filldraw[color = white,opacity=1] (\x, \y) rectangle (\x+1, \y+1);
\foreach \x/\y in {7/5 }
\filldraw[color = white,opacity=1] (\x, \y) rectangle (\x+1, \y+1);
\end{scope}
\begin{scope}[shift={(6,9)}]
\foreach \x/\y in {4/2, 5/2, 6/2,7/2 }
\filldraw[color = white,opacity=1] (\x, \y) rectangle (\x+1, \y+1);
\foreach \x/\y in {5/3, 6/3, 7/3 }
\filldraw[color = white,opacity=1] (\x, \y) rectangle (\x+1, \y+1);
\foreach \x/\y in {6/4,7/4 }
\filldraw[color = white,opacity=1] (\x, \y) rectangle (\x+1, \y+1);
\foreach \x/\y in {7/5 }
\filldraw[color = white,opacity=1] (\x, \y) rectangle (\x+1, \y+1);
\end{scope}
\begin{scope}[shift={(3,4)}]
\foreach \x/\y in { 3/1, 4/1, 5/1,6/1, 7/1}
\filldraw[color = white,opacity=1] (\x, \y) rectangle (\x+1, \y+1);
\foreach \x/\y in {4/2, 5/2, 6/2,7/2 }
\filldraw[color = white,opacity=1] (\x, \y) rectangle (\x+1, \y+1);
\foreach \x/\y in {5/3, 6/3, 7/3 }
\filldraw[color = white,opacity=1] (\x, \y) rectangle (\x+1, \y+1);
\foreach \x/\y in {6/4,7/4 }
\filldraw[color = white,opacity=1] (\x, \y) rectangle (\x+1, \y+1);
\foreach \x/\y in {7/5 }
\filldraw[color = white,opacity=1] (\x, \y) rectangle (\x+1, \y+1);
\end{scope}
\foreach \x/\y in { 8/10,9/10 ,10/10, 11/10, 12/10, 13/10, 14/10}
\filldraw[color = blue,opacity=.5] (\x, \y) rectangle (\x+1, \y+1);
\foreach \x/\y in { 8/9,9/9 ,10/9, 11/9, 12/9, 13/9, 14/9}
\filldraw[color = orange, opacity=.5] (\x, \y) rectangle (\x+1, \y+1);
\begin{scope}[shift={(2.5,-.5)}]
\filldraw[color = darkgreen, opacity=1] (10,10) rectangle (10.5,10.5);
\end{scope}
\begin{scope}[shift={(10,12)}, scale =0.5]
\foreach \x/\y in {-1/-3,0/-3,1/-3,2/-3, 3/-3, 4/-3, 5/-3,6/-3, 7/-3}
\filldraw[color = green,opacity=.6] (\x, \y) rectangle (\x+1, \y+1);
\foreach \x/\y in {0/-2,1/-2,2/-2, 3/-2, 4/-2, 5/-2,6/-2, 7/-2}
\filldraw[color = green,opacity=.1] (\x, \y) rectangle (\x+1, \y+1);
\foreach \x/\y in { 1/-1,2/-1,3/-1, 4/-1, 5/-1,6/-1, 7/-1}
\filldraw[color = green,opacity=.1] (\x, \y) rectangle (\x+1, \y+1);
\foreach \x/\y in { 2/0,3/0, 4/0, 5/0,6/0, 7/0}
\filldraw[color = green,opacity=.1] (\x, \y) rectangle (\x+1, \y+1);
\foreach \x/\y in { 3/1, 4/1, 5/1,6/1, 7/1}
\filldraw[color = green,opacity=.1] (\x, \y) rectangle (\x+1, \y+1);
\foreach \x/\y in {4/2, 5/2, 6/2,7/2 }
\filldraw[color = green,opacity=.1] (\x, \y) rectangle (\x+1, \y+1);
\foreach \x/\y in {5/3, 6/3, 7/3 }
\filldraw[color = green,opacity=.1] (\x, \y) rectangle (\x+1, \y+1);
\foreach \x/\y in {6/4,7/4 }
\filldraw[color = green,opacity=.1] (\x, \y) rectangle (\x+1, \y+1);
\foreach \x/\y in {7/5 }
\filldraw[color = green,opacity=.1] (\x, \y) rectangle (\x+1, \y+1);
\end{scope}
\begin{scope}[shift={(7,7)}, scale =0.5]
\foreach \x/\y in {0/-2,1/-2,2/-2, 3/-2, 4/-2, 5/-2,6/-2, 7/-2}
\filldraw[color = green,opacity=.3] (\x, \y) rectangle (\x+1, \y+1);
\foreach \x/\y in { 1/-1,2/-1,3/-1, 4/-1, 5/-1,6/-1, 7/-1}
\filldraw[color = green,opacity=.3] (\x, \y) rectangle (\x+1, \y+1);
\foreach \x/\y in { 2/0,3/0, 4/0, 5/0,6/0, 7/0}
\filldraw[color = green,opacity=.3] (\x, \y) rectangle (\x+1, \y+1);
\foreach \x/\y in { 3/1, 4/1, 5/1,6/1, 7/1}
\filldraw[color = green,opacity=.3] (\x, \y) rectangle (\x+1, \y+1);
\foreach \x/\y in {4/2, 5/2, 6/2,7/2 }
\filldraw[color = green,opacity=.3] (\x, \y) rectangle (\x+1, \y+1);
\foreach \x/\y in {5/3, 6/3, 7/3 }
\filldraw[color = green,opacity=.3] (\x, \y) rectangle (\x+1, \y+1);
\foreach \x/\y in {6/4,7/4 }
\filldraw[color = green,opacity=.3] (\x, \y) rectangle (\x+1, \y+1);
\foreach \x/\y in {7/5 }
\filldraw[color = green,opacity=.3] (\x, \y) rectangle (\x+1, \y+1);
\end{scope}
\begin{scope}[shift={(4,9)}, scale =0.5]
\foreach \x/\y in { 1/-1,2/-1,3/-1, 4/-1, 5/-1,6/-1, 7/-1}
\filldraw[color = green,opacity=.3] (\x, \y) rectangle (\x+1, \y+1);
\foreach \x/\y in { 2/0,3/0, 4/0, 5/0,6/0, 7/0}
\filldraw[color = green,opacity=.3] (\x, \y) rectangle (\x+1, \y+1);
\foreach \x/\y in { 3/1, 4/1, 5/1,6/1, 7/1}
\filldraw[color = green,opacity=.3] (\x, \y) rectangle (\x+1, \y+1);
\foreach \x/\y in {4/2, 5/2, 6/2,7/2 }
\filldraw[color = green,opacity=.3] (\x, \y) rectangle (\x+1, \y+1);
\foreach \x/\y in {5/3, 6/3, 7/3 }
\filldraw[color = green,opacity=.3] (\x, \y) rectangle (\x+1, \y+1);
\foreach \x/\y in {6/4,7/4 }
\filldraw[color = green,opacity=.3] (\x, \y) rectangle (\x+1, \y+1);
\foreach \x/\y in {7/5 }
\filldraw[color = green,opacity=.3] (\x, \y) rectangle (\x+1, \y+1);
\end{scope}
\begin{scope}[shift={(12,8)}, scale =0.5]
\foreach \x/\y in { 1/-1,2/-1,3/-1, 4/-1, 5/-1,6/-1, 7/-1}
\filldraw[color = green,opacity=.3] (\x, \y) rectangle (\x+1, \y+1);
\foreach \x/\y in { 2/0,3/0, 4/0, 5/0,6/0, 7/0}
\filldraw[color = green,opacity=.3] (\x, \y) rectangle (\x+1, \y+1);
\foreach \x/\y in { 3/1, 4/1, 5/1,6/1, 7/1}
\filldraw[color = green,opacity=.3] (\x, \y) rectangle (\x+1, \y+1);
\foreach \x/\y in {4/2, 5/2, 6/2,7/2 }
\filldraw[color = green,opacity=.3] (\x, \y) rectangle (\x+1, \y+1);
\foreach \x/\y in {5/3, 6/3, 7/3 }
\filldraw[color = green,opacity=.3] (\x, \y) rectangle (\x+1, \y+1);
\foreach \x/\y in {6/4,7/4 }
\filldraw[color = green,opacity=.3] (\x, \y) rectangle (\x+1, \y+1);
\foreach \x/\y in {7/5 }
\filldraw[color = green,opacity=.3] (\x, \y) rectangle (\x+1, \y+1);
\end{scope}
 \draw[thick, dashed] (11, 9) rectangle (14,10);
 \draw (12.5, 12) node {$T_1, T_2$};
 \draw (7, 10) node {A};
 \draw (15.4, 8.6) node {C};
 \draw (9.8, 7.3) node {D};
 \draw (14.2, 8.7) node {$\beta$};
 \draw (10, 8.7) node {$\alpha$};
 \begin{scope}[scale =1]
 \draw[color = white] (12.75, 9.75) node {$\star$};
\end{scope}
\end{tikzpicture}
\end{center}
 \vspace{-.6 cm}
\caption{The 2-scaled support in blue remains `the same', and therefore also $T_1$ and $T_2$.}\label{fig:proof1}
\end{figure}
By induction, we assume that if there is a \top (D), in the subsupport of $T_1$, then there is a \top (D') in the subsupport of $T_2$ (D' denotes iterated triangle by induction). Figure~\ref{fig:proof2} shows that a new `0' cell appears if and only if the condition in Claim 1 is satisfied; if the number is smaller, then the neighboring triangles (A and D in the picture) will see that there is a `1' in $\support (T_1)$ if and only if there is a `1' in both corresponding cells in $\support (T_2)$. This proves the equivalence.
\begin{figure} [ht!]
%\vspace{.5 cm}
\begin{center}
\begin{tikzpicture} [scale = 0.5]
\foreach \x/\y in { 8/10,9/10 ,10/10, 11/10}
\filldraw[color = blue,opacity=.5] (\x, \y) rectangle (\x+1, \y+1);
\foreach \x/\y in { -1/9, 0/9 }
\filldraw[color = orange, opacity=.5] (\x, \y) rectangle (\x+1, \y+1);
\foreach \x/\y in { 3/9, 4/9 ,5/9}
\filldraw[color = orange, opacity=.5] (\x, \y) rectangle (\x+1, \y+1);
\foreach \x/\y in { 8/9, 9/9 ,10/9, 11/9}
\filldraw[color = orange, opacity=.5] (\x, \y) rectangle (\x+1, \y+1);
\begin{scope}[shift={(2.5,-.5)}]
\filldraw[color = darkgreen, opacity=1] (2,10) rectangle (2.5,10.5);
 \draw[color = white] (2.25,10.25) node {$\star$};
\end{scope}
 \draw[thick, dashed] (3, 9) rectangle (6,10);
 \begin{scope}[shift={(2.5,-.5)}]
 \filldraw[color = white, opacity=.8] (7.5,10.5) rectangle (8.5,11.5);
\filldraw[color = lightgreen, opacity=.4] (7,10) rectangle (8.5,10.5);
\filldraw[color =  lightgreen, opacity=.4] (8,11) rectangle (8.5,11.5);
\filldraw[color =  lightgreen, opacity=.4] (7.5,10.5) rectangle (8.5,11);
 \end{scope}
  \draw[color= white, step=.5cm, thin, dotted] (-1, 9) grid (1,10);
\draw[color= white,step=.5cm, thin, dotted] (3, 9) grid (6,10);
\draw[color= white,step=.5cm, thin, dotted] (8, 9) grid (12,11);

 \end{tikzpicture}
\end{center}
\caption{A star is born if and only if the Claim 1 criterion is $\Delta(L, R)-1$.}\label{fig:proof2}
\end{figure}
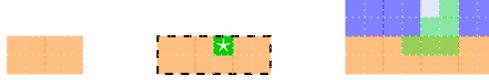
\vspace{-.2 cm}
\end{proof}

This proof explains where the newborn stars appear in the iteration of new diagrams. In going to the limit, we will get infinitely many new single `0'-cells in each bounded region of only `1's, and each such `0' will give birth to a new single `0', so we think of this as some fractal behavior, rather than convergence. But, we do have partial convergence for CA-safe discrete triangles to real ones (those limit triangles are not play-triangles in games defined here); see also Figure~\ref{fig:limgame}, and moreover we obtain a certain converging `tail'. 

\begin{theorem}
Suppose that $T_1 = T_1(x,y,h)$ is a $\CA_{L,R}$-safe triangle. Then, by iterating the construction in Theorem~\ref{thm:2} (and re-indexing the triangles), in the limit diagram, $\lim \CA^n/2^n$, $\lim T_n$ is a real right justified right-angle triangle with base and height $h+1$, with the right angle at $(x,y-1)$. If there is a newborn star re-scaled at $y-1$, below a triangle $s$ at level $y$, then there will be another newborn star in the next iteration at level $y-3/2$. In the limit diagram the sequence converges to level $y-3$ and distance $2L$ to the nearest triangle $s'$ to the left just below $\support(s)$.
\end{theorem}
\begin{proof}
The first part follows by iterating Theorem~\ref{thm:2} and noting that $1/2+1/4+1/8+\cdots =1$. The second part follows by generalizing the 5 cells below the dashed line below E in Figure~\ref{fig:proof1}. Indeed, they satisfy $2L+2R+1=\Delta (2L,2R)-1=\Delta(4L,4R)/2$, and so we may iterate the newborn star argument in the proof of Theorem~\ref{thm:2}. Note that, even if triangle $C$ were one unit smaller in the first diagram (and it could not be 2 units smaller by the assumption), then the extension of $\base(C')$ below and one unit to the left of $\base(C)$ would imply the required $2L+2R+1$ `1'-cells. It follows that the converging sequence is 
%$y-3/2,y-9/4,y-21/8,y-45/16\ldots =
$(\alpha+2L+\frac{2}{2^n}, y-\frac{3\cdot 2^n-3}{2^n})$, where $C'=T(\alpha,\cdot, \cdot)$.
\end{proof}
\noindent {\bf Acknowledgements:} Two persons inspired this paper. I thank Matthew Cook for many discussions of blocking maneuvers in combinatorial games and also for contributing to the definition of the CA. I would also like to thank my grandmother Signe Classon who used to play a game with a blocking maneuver with us in the 1970s, the traditional Swedish game of ``f\"orbju\!\' \,namn". I also would like to thank Melissa Huggan, Dalhousie University, for a much appreciated contribution in realizing this submission. Finally I thank the anonymous referees for their comments and questions. 
%\clearpage

\end{document}